\documentclass[12pt]{article}
\usepackage{amsmath,amsfonts,amsthm,amscd,amssymb,graphicx}

\usepackage{graphicx}

\usepackage{amssymb}
\usepackage[colorlinks,citecolor=blue,linkcolor=blue]{hyperref}
\usepackage{xcolor}

\def\rit{{\Bbb R}}
\def\cit{{\Bbb C}}

\def\tit{{\Bbb T}}

\def\eps{\varepsilon}

\def\Remarks{{\noindent \it Remarks. }} 
\def\Remark{{\noindent \it Remark. }}

\newtheorem{theorem}{Theorem}[section]
\newtheorem{lemma}[theorem]{Lemma}
\newtheorem{e-proposition}[theorem]{Proposition}

\newtheorem{e-definition}[theorem]{Definition\rm}

\newtheorem{theoreme}{Th\'eor\`eme}[section]

\newtheorem{proposition}[theoreme]{Proposition}

\def\og{\leavevmode\raise.3ex\hbox{$\scriptscriptstyle\langle\!\langle$~}}
\def\fg{\leavevmode\raise.3ex\hbox{~$\!\scriptscriptstyle\,\rangle\!\rangle$}}

\def\beq{\begin{equation}}
\def\eeq{\end{equation}}

\begin{document}

\centerline{\Large \bf Asymptotic behaviour of solutions of linearized}

\bigskip

\centerline{\Large \bf   Navier Stokes equations  in the long waves regime}

\bigskip

\centerline{D. Bian\footnote{Beijing Institute of Technology, School of Mathematics and Statistics, Beijing, China}, 
E. Grenier\footnote{UMPA, CNRS  UMR $5669$, Ecole Normale Sup\'erieure de Lyon, Lyon, France}}



\subsubsection*{Abstract}


The aim of this paper is to describe the long time behavior of solutions of linearized Navier Stokes equations near a concave shear
layer profile in the long waves regime, namely for small horizontal Fourier variable $\alpha$, when the viscosity $\nu$ vanishes.
We show that the solutions converge exponentially  to $0$, except in some range of $\alpha$, namely for 
$\nu^{1/4} \lesssim |\alpha| \lesssim \nu^{1/6}$, where there exists one unique unstable mode, with an associated eigenvalue $\lambda$,
such that $\Re \lambda$ is of order $\nu^{1/4}$.
In this regime we give a complete description of the solutions of linearized Navier Stokes equations
as the sum of the projection over  the unique exponentially growing mode
and of an exponentially decaying term.
The study of this linear instability is a key point in the study of the nonlinear instability of Prandtl bounday layers and of shear layer profiles
\cite{Bian6}.


\section{Introduction}


In this paper we study the long time behaviour of solutions of 
the linearized incompressible Navier Stokes equations in the half plane $\rit \times \rit_+$,
near a shear flow $U$, namely of solutions $v$ to 
 \beq \label{NS1} 
\partial_t v + (U \cdot \nabla) v +  (v \cdot \nabla) U - \nu \Delta v + \nabla p = 0,
\eeq
\beq \label{NS2}
\nabla \cdot v = 0 ,
\eeq
together with the Dirichlet boundary condition
\beq \label{NS3} 
v = 0 \qquad \hbox{for} \qquad y = 0
\eeq
and initial data $v(0,\cdot) = v_0(\cdot)$.

In all this paper, we will assume that $U(y) = (U_s(y),0)$, where $U_s$ is a given concave analytic function, with $U_s(0) = 0$, $U_s'(0) \ne 0$ and
 such that $U_s(y)$ converges exponentially fast to some constant $U_+ > 0$ when $y$ goes to $+ \infty$.
This in particular includes the exponential profile $U_s(y) = U_+ (1 - e^{-y})$.
Note that, according to Squire's theorem, the three dimensional stability problem reduces to the two dimensional one, thus, as long as we
consider  the {\it linear} stability, there is no difference between the three dimensional and the two dimensional cases. 

In the recent years, many breakthroughs \cite{Helffer,Bed1,Chen},
have been done in the study of these solutions in the periodic setting $x \in \tit$, namely 
for horizontal wave numbers $\alpha$ away from $0$, and in particular in the study of the so called {\it enhanced dissipation}: because of the
interaction between transport and diffusion, solutions of linearized Navier Stokes equations
 decay faster than that of the heat equation.

In this paper we focus on the opposite case, namely on small $| \alpha |$, for which a complete different phenomena, namely 
{\it instability}, occurs: instead
of being damped to $0$, some solutions increase exponentially fast.
This instability is fundamental in the understanding of the behavior of shear layers at high Reynolds number (i.e. low viscosity $\nu$) and of the
transition from laminar flows to turbulent ones. Its study in physics began at the end of the nineteenth century, when Reynolds and Rayleigh
began to investigate the stability of inviscid shear layers (with $\nu = 0$). During the first half of the twentieth century, Prandtl, followed by Orr, Sommerfeld,
Tollmien, Schlichting and C.C. Lin, to quote only a few, progressively understood that any non trivial shear layer was linearly unstable in the small
viscosity regime (see for instance \cite{Drazin},\cite{Reid} for an historical presentation and formal computations).

\medskip

As the viscosity goes to $0$, two kinds of instabilities appear, depending on whether the shear flow $U$ is linearly stable when $\nu = 0$ or not:

\begin{itemize}

\item Some flows are spectrally unstable for linearized Euler equations (namely for $\nu = 0$). According to Rayleigh's criterium, this implies that
they have an inflection point. In this case they remain linearly unstable provided the viscosity is small enough. They are even nonlinearly unstable
\cite{GN2}: unstable modes grow with a speed $O(1)$, leading to nonlinear perturbations which reach a size $O(1)$ in $L^\infty$. Their typical size 
in $x$ and $y$ is also $O(1)$.
 
 \item Other flows are spectrally stable for linearized Euler equations, for instance concave flows, which do not have inflection points.
 However, even for these flows, provided $\nu$ is small enough, there exists unstable modes for the linearized Navier Stokes equations.
 That is, adding viscosity {\it destabilizes} the flow. The existence of such unstable modes have been mathematically shown in \cite{GGN3}.
 They are characterized by a very slow growth rate, of order $O(\nu^{1/4})$, much slower than the growth rate of the first kind of instabilities.
 Their typical size in $x$ is also very different. It is very large, of order $O(\nu^{-1/4})$ to $O(\nu^{-1/6})$. We thus face long wave instabilities,
 growing very slowly.
Many related questions, like for instance their nonlinear evolution,  remain open \cite{Bian5,Bian6}.
\end{itemize}

\begin{figure} \label{figure1}
\centerline{\includegraphics[width=7cm]{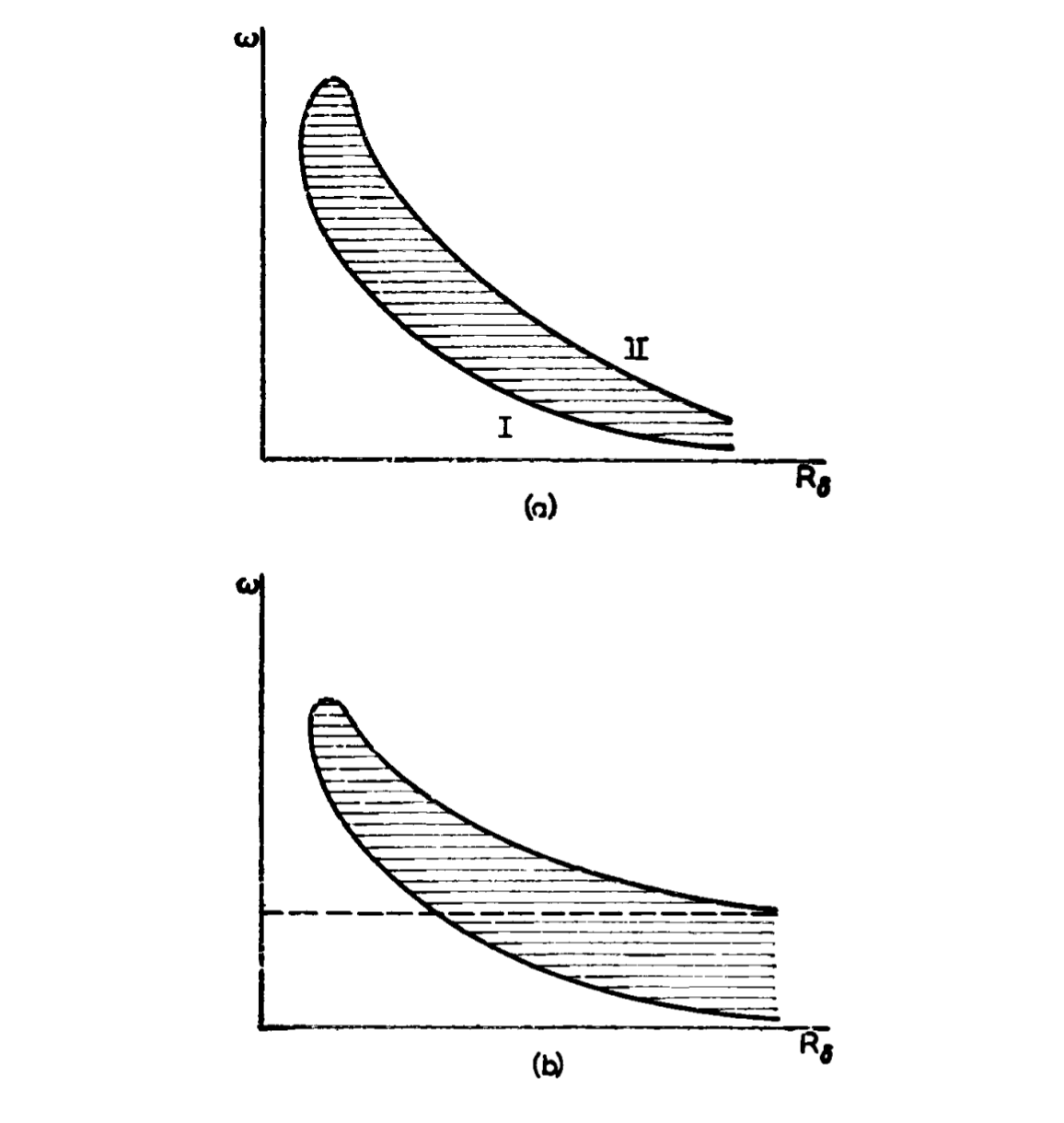}}
\caption{Area of linear instability (grey area) as a function of the Reynolds number (horizontal axis) and horizontal wave number $\alpha$ 
(vertical axis, called $\omega$), taken from L. Landau and E. Lifschitz's "Fluid Mechanics" \cite{Landau}, page $170$. Figure (a) corresponds to
our second case (no inflection point) and figure (b) to our first case (instable for Euler).   }
\label{rawdata}
\end{figure}

The two cases are clearly displayed on figure $8.8$ of the classical  source book \cite{Drazin}. We also refer to 
L. Landau's and E. Lifschitz' classical textbook \cite{Landau} (page $170$, figure $29$).
For the second kind of flows, 
all the physics of the linear instability of shear flows for Navier Stokes equations lies in the area $| \alpha | \ll 1$. In particular, it is missed
as soon as we are periodic in $x$, namely in domains of the form $\tit \times \rit_+$ or $\tit \times [0,1]$, even with a large periodicity in $x$,
since, as is clear on the figures previously quoted, there is no instability in the area $| \alpha | > \alpha_0$ provided $\nu$ is small enough
with respect to $\alpha_0$.

Namely, the horizontal periodicity of  unstable modes increase when the Rey\-nolds number increases, or equivalently the horizontal wave numbers
of unstable modes go to $0$ as the viscosity goes to $0$ (see Figure $8.8$ of  \cite{Drazin}).

In this paper we focus on this second kind of instabilities. We greatly simplify and precise \cite{GGN3}, and completely study
the evolution of  linearized solutions through a complete description of the Green function of linearized Navier Stokes equations.
This opens the way to the study of the nonlinear instability of generic shear layer profiles for the nonlinear Navier Stokes equations
and for Prandtl boundary layers \cite{Bian5,Bian6}. 

\medskip

Let us now detail our main result.
We define the linearized Navier Stokes operator $L_{NS}$ to be
$$
L_{NS} (v) =  P \Bigl( (U \cdot \nabla) v +  (v \cdot \nabla) U - \nu \Delta v  \Bigr)
$$
together with $\nabla \cdot v = 0$ and Dirichlet boundary condition, where $P$ is Leray's projection on divergence free vector fields.
We take the Fourier transform in $x$, with dual Fourier variable $\alpha$. We denote by $L_{NS,\alpha}$
the Fourier transform of $L_{NS}$ and by $v_\alpha(t)$ the Fourier transform of the solution $v(t)$ of (\ref{NS1})-(\ref{NS3}) in such a way that
\beq \label{eqvalpha}
\partial_t v_\alpha + L_{NS,\alpha} v_\alpha = 0 ,
\eeq
with initial data 
\beq \label{eqvalpha2}
v_\alpha(0,\cdot) = v_{\alpha,0}(\cdot),
\eeq
where $v_{\alpha,0}$ is the Fourier transform in $x$ of the initial data $v_0$.
The main result of this paper is the following Theorem.

\begin{theorem} \label{maintheo}
Let $U_s$ be a concave and analytic shear layer profile, such that $U_s(0) = 0$, $U_s'(0) \ne 0$,  such that $U_s(y)$ converges exponentially
fast to some constant $U_+ > 0$ as $y \to + \infty$, and such that $U_s''$ converges exponentially fast to $0$ as $y \to + \infty$, namely such that
$$
| U_s''(y) | \le C_\beta e^{-\beta \, \Re y}
$$
for every $y$ and for some  positive constant $\beta$. 
Then there exists four functions $\alpha_\pm^{def}(\nu)$ and $\alpha_\pm(\nu)$ with
$$
\alpha_-^{def}(\nu) < \alpha_-(\nu) < \alpha_+(\nu) < \alpha_+^{def}(\nu),
$$
four different positive constants $C_\pm^{def}$ and $C_\pm$, a function $\lambda(\alpha,\nu)$,
and a projection operator $P_{\alpha,\nu}$, such that, for any initial data $v_0$ and any small enough $\alpha$ and $\nu$,

\begin{itemize}

\item[i)]  if  $| \alpha | \le \alpha_-^{def}(\nu)$
or $| \alpha | \ge \alpha_+^{def}(\nu)$ then 
\beq \label{project1}
P_{\alpha,\nu} v_{\alpha,0} = 0,
\eeq
\item[ii)] if $\alpha_-^{def}(\nu) \le | \alpha | \le \alpha_+^{def}(\nu)$, then
\beq \label{project2}
P_{\alpha,\nu} v_{\alpha,0} = e^{\lambda(\alpha,\nu) t} \Bigl( v_{\alpha,0} | v^\star(\alpha,\nu) \Bigr)_{L^2} v(\alpha,\nu) ,
\eeq
where $v(\alpha,\nu)$ is an eigenvector of $L_{NS,\alpha}$ with corresponding eigenvalue $\lambda(\alpha,\nu)$,
and where $v^\star(\alpha,\nu)$ is an eigenvector of the adjoint operator $L_{NS,\alpha}^\star$ 
with the same eigenvalue, and unit norm.
\end{itemize}
Moreover, 
as $\nu \to 0$,   
\beq 
\alpha_-^{def}(\nu) \sim C_-^{def} \nu^{1/4}, \qquad \alpha_+^{def}(\nu) \sim C_+^{def} \nu^{1/6},
\eeq
\beq 
\alpha_-(\nu) \sim C_- \nu^{1/4}, \qquad \alpha_+(\nu) \sim C_+ \nu^{1/6}.
\eeq
When $\alpha_-(\nu) < |\alpha| < \alpha_+(\nu)$ then 
\beq
 \Re \lambda(\alpha, \nu) > 0,
\eeq
and
\beq \label{maxmagnitude}
\sup_\alpha  | \Re \lambda(\alpha,\nu) | \sim C_1 \nu^{1/2} 
\eeq
for some positive constant $C_1$. On the contrary, for $|\alpha| \notin [\alpha_-(\nu),\alpha_+(\nu)]$, $\Re \lambda(\alpha,\nu) < 0$.

Moreover, $v_\alpha(t)$, solution of (\ref{eqvalpha}) with initial data $v_\alpha(0,\cdot) = v_{\alpha,0}(\cdot)$, satisfies
\beq \label{asymptotic}
v_\alpha (t,y) =  P_{\alpha,\nu} v_{\alpha,0}(y)
+ \int_0^{+\infty} G_{\alpha,\nu}(t,x,y) \, v_{\alpha,0}(x) \, dx,
\eeq
where the Green function $G_{\alpha,\nu}(t,x,y)$ satisfies
$$
| G_{\alpha,\nu}(t,x,y) | \lesssim  C \eta(\alpha,\nu) \exp( -  \theta(\alpha,\nu) \langle t \rangle).
$$
In this formula, if $\alpha_-^{def}(\nu) < |\alpha| < \alpha_+^{def}(\nu)$, then
the decay rate $\theta(\alpha,\nu)$ equals
$$
\theta(\alpha,\nu) = C_\theta  \nu^{1/3} \alpha^{2/3}
$$
and  the amplification factor $\eta(\alpha,\nu)$ equals
$$
\eta(\alpha,\nu) = C_\eta {\alpha^{5/3} \over \nu^{2/3}} 
$$
for some positive constants $C_\theta$ and $C_\eta$. 

If, on the contrary, $|\alpha| < \alpha_-^{def}(\nu)$ or $|\alpha| > \alpha_+^{def}(\nu)$, then $\theta(\alpha,\nu) = C_\theta \nu^{1/2}$ 
and $\eta(\alpha,\nu) = C_\eta \nu^{-1/4}$.
\end{theorem}

In other terms, in the range $[\alpha_-(\nu),\alpha_+(\nu)]$, the solution $v_\alpha(t)$ is the sum of an exponentially growing mode,
with slow growth rate $\Re \lambda(\alpha,\nu)$, and of an exponentially decaying remainder, with slow decay rate $\theta(\alpha,\nu)$.
In the region $[\alpha_-^{def}(\nu),\alpha_+^{def}(\nu)]$, the operator $L_{NS}$ has only one eigenvalue with an imaginary part greater than 
$\theta$. This eigenvalue is the only unstable eigenvalue between $\alpha_-(\nu)$ and $\alpha_+(\nu)$.  
Much more detailed descriptions of the eigenvalue $\lambda(\alpha,\nu)$ and of the Green function $G_{\alpha,\nu}$ 
are given in sections $7$ and $8$.  
 In particular the numerical values of the constant $C_-$ and $C_+$ are given in formula
(\ref{branch1}) and (\ref{branch2}).
The projection on the adjoint mode is detailed in section \ref{adjointNS}.

The concavity assumption may be relaxed, up to additional technicalities.
Note that a similar theorem holds true if $U_s'(0) = 0$ but $U_s''(0) \ne 0$, with exponents $1/4$ and $1/10$ instead of $1/4$ and $1/6$ 
(see \cite{Drazin}, \cite{Reid} and \cite{Benoit}). 

\medskip

This result proves the existence of a "spectral gap" $\theta(\alpha,\nu)$, which is of order $\nu^{1/2}$ when $\alpha(\nu) \sim \nu^{1/4}$, 
and of order $\nu^{4/9}$  when $\alpha(\nu) \sim \nu^{1/6}$. This spectral gap is thus always smaller than the real part of the unstable mode given by
(\ref{maxmagnitude}).  

Note that $\eta(\alpha,\nu)$ can be interpreted as an amplification factor. As linearized Navier Stokes operator is not self-adjoint, solutions
may have a transient increase before decreasing exponentially. This amplification factor $\eta(\alpha,\nu)$ is of magnitude at most $\nu^{-1/4}$.

\medskip

To prove this theorem we directly study the Green's function of linearized Navier Stokes equations, and give explicit bounds on it.
The plan of the paper is as follows.
The second part of the paper is devoted to the introduction of the classical Rayleigh and Orr Sommerfeld equations, the third part to the
detailed construction of solutions for Rayleigh equation, the fourth part to the investigation of Airy equation and to a first inversion of
Orr Sommerfeld operator in the case of fast decaying source terms, leading to the construction of "fast solutions" to Orr Sommerfeld equations. 
In the fifth part we construct "slow" solutions
of Orr Sommerfeld equations and   build an approximate Green function for Orr Sommerfeld equations in section $6$.
Section $7$ is devoted to the investigation of  the relation dispersion of linearized Navier Stokes equations, section $8$ to a brief
study of the adjoint of Orr Sommerfeld, 
and  section $9$ to a complete study of the genuine Green function of $L_{NS,\alpha}$,
which ends the proof of Theorem \ref{maintheo}. The Appendix is devoted to Airy and Tietjens functions, and to an iteration procedure which
is used several times in this article.

\subsubsection*{Notations}

In all this paper we will denote 
$$
\langle t \rangle = 1 + | t |,
$$
and we will say that $f \lesssim g$ if there exists some positive constant such that $|f| \le C g$.

In all the paper we will consider functions which are holomorphic in a "pencil like" neighborhood of $\rit_+$, namely,
functions which are holomorphic on 
$$
\Gamma_{\beta_0,\beta_1} = \Bigl\{ z \in \cit \quad | \quad | \Im z | \le \min(\beta_0 \, \Re z, \beta_1), \quad \Re z \ge 0 \Bigr\}
$$
for some positive constants $\beta_0$ and $\beta_1$.

We will consider functions which have a logarithmic branch at some $y_c \in \Gamma_{\beta_0,\beta_1}$. 
We thus introduce
$$
\Gamma_{\beta_0,\beta_1}(y_c) = \Gamma_{\beta_0,\beta_1} - \Bigl\{ y_c - i \rit_+ \Bigr\}.
$$
Let $\sigma > 0$.
We define $X_\sigma$ to be the space of functions $f$ which are holomorphic on $\Gamma_{\beta_0,\beta_1}$ and such that
$$
\| f \|_{X_\sigma} = \sup_{y \in \Gamma} | f(y) |  e^{| \sigma | y} < + \infty.
$$
Similarly we define $X_\sigma(y_c)$ to be the space of functions which are holomorphic on $\Gamma_{\beta_0,\beta_1}(y_c)$ 
with the same norm.

Other function spaces will be defined in section \ref{further}.


\section{Orr Sommerfeld equations \label{part2}}


Let us first introduce the classical Orr Sommerfeld and Rayleigh equations.
  We refer to \cite{Reid} and \cite{GN} for more details on all these aspects.
We recall that the solution $v(t)$ of linearized Navier Stokes equations 
is given by Dunford's formula
\beq \label{contour1}
v(t,x) = {1 \over 2 i \pi} \int_\Gamma e^{\lambda t} (L_{NS} + \lambda)^{-1} v_0 \; d\lambda
\eeq
in which $\Gamma$ is a contour on the "right" of the spectrum.
 We are therefore led to study the resolvent of $L_{NS}$, namely to
study the equation
\beq \label{resolvant}
(L_{NS} + \lambda) v = f,
\eeq
where $f$ is a given forcing term and $\lambda$ a complex number.
To take advantage of the divergence free condition, we introduce the stream function $\psi$ of $v$ and take its Fourier
transform in $x$, with dual variable $\alpha$, which leads to look for solutions of (\ref{resolvant}) of the form
$$
v = \nabla^\perp \Bigl( e^{i \alpha x  } \psi(y) \Bigr) .
$$
According to the traditional notation \cite{Reid}, we introduce the complex number $c$, defined by  
$$
\lambda = - i \alpha c .
$$
We also take the Fourier transform of the forcing term $f$
$$
f = \Bigl( f_1(y),f_2(y) \Bigr) e^{i \alpha x  } .
$$
Taking the curl of (\ref{resolvant}), we get the classical Orr Sommerfeld equation
\beq \label{Orrmod0}
Orr_{\lambda,\alpha,\nu}(\psi) =  (U_s - c)  (\partial_y^2 - \alpha^2) \psi - U_s''  \psi  
- { \nu \over i \alpha}   (\partial_y^2 - \alpha^2)^2 \psi =  i {\nabla \times f \over \alpha}
\eeq
where
$$
\nabla \times (f_1,f_2) = i \alpha f_2 - \partial_y f_1.
$$
The Dirichlet boundary condition
gives 
\beq \label{condOrr}
\psi(0) = \partial_y \psi(0) = 0.
\eeq
Let us define
\beq \label{epsilon}
\eps = {\nu \over i \alpha} .
\eeq
We note that when $\nu$ goes to $0$, Orr Sommerfeld equation degenerates into the classical Rayleigh equation
\beq \label{Rayleigh}
   (U_s - c)  (\partial_y^2 - \alpha^2) \psi - U_s''  \psi  
=  i {\nabla \times f \over \alpha}.
\eeq
It is in fact more interesting to define
\beq \label{tildec}
 \tilde c = c - 2 \eps \alpha^2,
\eeq 
and the modified Rayleigh operator by
\beq \label{Rayleigh2}
Ray \, \psi = (U_s - \tilde c) (\partial_y^2 - \alpha^2) \psi - V \psi
\eeq
where
\beq \label{defiV}
V = U_s'' + \eps \alpha^4 .
\eeq
Note that the modified Rayleigh operator and the genuine one coincide when $\nu = 0$.
From now one we drop the tilde from $\tilde c$.
Note also that Orr Sommerfeld equation may be seen as a small viscous perturbation of the modified Rayleigh's equation since
\beq \label{introDiff} 
Orr_{\alpha,c,\nu} = Ray - \eps \partial_y^4.
\eeq
This perturbation is very singular since, in the regime we consider,
 both $\eps \to 0$ and $\min(U_s - c) \to 0$, thus $Orr_{\alpha,c,\nu}$ degenerates to $-U_s'' \psi$, loosing four derivatives
 in this limiting process.
 
 Orr Sommerfeld equation may also be considered as a perturbation of the classical Airy equation. Namely let
 \beq \label{defiAiryop}
 {\rm Airy}(\psi) = (U_s -  c) \psi - \eps \partial_y^2 \psi
 \eeq
 and
 $$
 {\cal A} = {\rm Airy} \circ \partial_y^2  .
 $$
 Then
 \beq \label{introAiry}
 Orr_{\alpha,c,\nu} = {\cal A} - {\cal A}_0
 \eeq
 where ${\cal A}_0$ is the zeroth order, multiplication operator
 $$
 {\cal A}_0  = U_s'' + \eps \alpha^4   .
 $$
 Part of the analysis is to play between the two decompositions (\ref{introDiff}) and (\ref{introAiry}).
 The first decomposition considers $- \eps \partial_y^4$ as a perturbation of the second order Rayleigh operator, and is interesting away
 from the so called "critical layer" $y_c$ defined by
 $$
 U_s(y_c) = c.
 $$ 
 The second one considers zeroth order terms as a perturbation of second and fourth order derivatives,
 and is interesting near the "critical layer", where there is a $(y-y_c) \log(y - y_c)$ singularity.

\medskip

A key point of this work is to construct four independent solutions to the Orr Sommerfeld equations, without taking into account
their boundary conditions.
A first step is to study their behavior at infinity. Asymptotically, solutions of Orr Sommerfeld equations should behave like
exponentials $e^{-\mu y}$, where $\mu$ satisfies
\beq \label{disperinfini}
\Bigl[ U_s - c - \eps (\mu^2 - \alpha^2) \Bigr] (\mu^2 - \alpha^2) = 0.
\eeq
Thus $\mu = \pm \mu_s$ or $\pm \mu_f$ with
\beq \label{disperinfini2}
\mu_s = \alpha, \qquad \mu_f = \sqrt{\alpha^2 + {i \alpha \over \nu } (U_+ - c) }.
\eeq
We note that $\Re \mu_s$ and $\Re \mu_f$ do not change sign provided $c$ remains small enough and that $| \mu_f | \gg 1$.
As a consequence, there exists two independent solutions which behave "slowly" and two independent solutions which display
a "fast" behaviour. The "slow" solutions can be constructed starting from solutions of Rayleigh equation, namely neglecting $- \eps \partial_y^4$
and using (\ref{introDiff}),
and the "fast" solutions" can be constructed starting from solutions of Airy equation, namely neglecting zeroth order terms
and using (\ref{introAiry}).


\section{Linearized Euler equations \label{studyRay0}}



\subsection{Asymptotic behavior}


The aim of this section is to derive optimal bounds on solutions of linearized Euler equations in the case of monotonic concave profiles 
through a detailed analysis of Rayleigh equation, as a warm up for the Navier Stokes case.
Let $v$ be a solution to the linearized Euler equations
\beq \label{linearEuler1}
\partial_t v + (U \cdot \nabla) v + (v \cdot \nabla) U + \nabla q = 0,
\eeq
\beq \label{linearEuler2}
\nabla \cdot v = 0
\eeq
with initial data $v(0,\cdot) = v_0(\cdot)$ and classical boundary conditions $v \cdot n = 0$ at $y = 0$, where
$n = (0,1)$, and $v \to 0$ as $y \to + \infty$.
As usual we take the Fourier transform in $x$, with dual variable $\alpha$, and introduce the stream function $\psi_\alpha(t,x)$.
We will prove the following result.

\begin{theorem} \label{growthRayleigh}
Assume that $U_s(y)$ is a concave and analytic profile, with $U_s(0) = 0$, $U_s'(0) \ne 0$,
which converges exponentially fast to some positive constant $U_+$ as $y \to + \infty$, and such that $U_s''$ converges
exponentially fast to $0$ as $y \to + \infty$. We moreover assume that there exists no purely imaginary eigenvalue to Rayleigh operator.
Then, provided $| \alpha |$ is small enough, we have
\beq \label{result1}
| \alpha | \, \| \psi_\alpha(t, \cdot) \|_{L^\infty} + \| \partial_x \psi_\alpha(t, \cdot) \|_{L^\infty}  = O(\langle \alpha t \rangle^{-1}).
\eeq
\end{theorem}

Note that we have only a slow decay, in $\langle t \rangle^{-1}$.
This result has already been proved for general, non holomorphic profiles  in \cite{Bed2,Bian3,Ionescu1,Wei}.
As $U_s(y)$ is monotonic and holomorphic, the proof turns out to be much simpler.
It is based on a careful analysis of
Rayleigh's equation without forcing term and without boundary condition
\beq \label{Rayleighwithout}
(U_s - c) (\partial_y^2 - \alpha^2) \psi - U_s'' \psi = 0 .
\eeq
We will only consider positive $\alpha$, the other case being identical.
At infinity, if $c \ne U_+$, as $U_s''$ decays exponentially fast to $0$,  there exist two solutions $\psi_{+,\alpha,c}$ 
and $\psi_{-,\alpha,c}$ which asymptotically  behave  like $e^{\alpha y}$ and $e^{- \alpha y}$.

As $U_s$ is monotonic, it is well known that there exists no eigenvalue $c$ with $\Im c > 0$ (so called Rayleigh's inflection point
theorem \cite{Reid}). Hence the essential spectrum of Rayleigh's operator reduces to the range of $U_s$ and the point spectrum
reduces to a finite number of possible imaginary eigenvalues
\beq \label{spectrumRayleigh}
\sigma_{Rayleigh} = Range(U_s) \cup \sigma_P .
\eeq
To simplify the analysis, we assume that $\sigma_P = \emptyset$.
The next section is devoted to the construction of particular solutions for Rayleigh's equation, which are then used to construct
an explicit Green function for this equation. Then, the study of Dunford's integral leads to the desired decay on $\psi_\alpha(t)$.
The key point of this approach is to study the singularity of the Green function in $y$ and $c$.


\subsection{Study of Rayleigh's equation} 


\begin{theorem} \label{Rayexp}
For small $|\alpha|$, there exists two independent solutions $\psi_{\pm,\alpha}(y)$ of Rayleigh equation, defined
on  $\Gamma_{\beta_0,\beta_1}(y_c)$, with unit Wronskian, which behave like
$O(e^{\pm | \alpha | y})$ when $y \to + \infty$. Near $y_c$, they can be written under the form
\beq \label{psismally}
\psi_{\pm,\alpha}(y) = P_{\pm,\alpha}(y - y_c,c) + (y - y_c) \log(y - y_c) Q_{\pm,\alpha}(y - y_c,c) 
\eeq
where $P_{\pm,\alpha}$ and $Q_{\pm,\alpha}$ are holomorphic functions of their two arguments, and
where $y_c$ is the critical layer, defined by $U_s(y_c) = c$.
Moreover, 
\beq \label{sizesingularity}
Q_{-,\alpha} = O(c), \qquad Q_{+,\alpha} = O(1), 
\eeq
and, for small $\alpha$ and $c$, 
\beq \label{psippsi}
{\psi_{-,\alpha}'(0) \over \psi_{-,\alpha}(0)} = - {U_s'(0) \over c} - {\alpha \over c^2}  (U_+ - c)^2
+ {\alpha^2 \over c^2} (U_+ - c)^4 \Omega_0(0,c) + O \Bigl( { \alpha^3 \over c^2} \Bigr) ,
\eeq  
where
\beq \label{defiOmega0}
\Omega_0(0,c) =  -{1 \over ( U_+ - c)^2} \int_0^{+ \infty} \Bigl[
{(U_s(z) - c)^2 \over (U_+ - c)^2} - {(U_+ - c)^2 \over (U_s(z) - c)^2} \Bigr] dz .
\eeq
\end{theorem}

\Remarks
The vorticity $\omega_{\pm,\alpha}$ of $\psi_{\pm,\alpha}$ is directly given by Rayleigh equation through
\beq \label{vorticityEuler}
\omega_{\pm,\alpha} = {U_s'' \over U_s - c} \psi_{\pm,\alpha}.
\eeq
In particular, it has a $(y - y_c)^{-1}$ singularity at $y_c$. It is exponentially rapidly decreasing at infinity, like $U_s''(y)$, namely like $e^{- \beta y}$, with a rate
independent on $\alpha$, in strong contrast with $\psi_{\pm,\alpha}$, whose decay rate goes to $0$ as $\alpha \to 0$:
the vorticity decays faster than the velocity, which only decays like $e^{- |\alpha| y}$. This vertical scale $\alpha^{-1}$ corresponds to a large scale
 recirculation of the fluid, through large structures of horizontal and vertical scales of order $\alpha^{-1}$.

As stated in (\ref{sizesingularity}),
the singularity of $\psi_{-,\alpha}$ at $y_c$ is of magnitude $O(c)$ only, a remark which will be used in section \ref{approximateinterior}.
On the contrary the singularity of $\psi_{+,\alpha}$ at $y_c$ is of magnitude $O(1)$.
Note that $\psi_{\pm,\alpha}$ has a logarithmic branch at $y_c$ and is thus defined on $\Gamma_{\beta_0,\beta_1}$ deprived of some
segment ending on $y_c$, namely on $\Gamma_{\beta_0,\beta_1}(y_c)$.

\begin{proof}
The proof is classical and is just a rephrasing of the computations detailed in \cite{Drazin}, \cite{Drazin2} or \cite{Reid}. 
We note that  (\ref{Rayleighwithout}) has a singular point at $y_c$. We first construct the solutions locally near $y_c$.
We  expect to have two independent 
solutions, called $\psi_A$ and $\psi_B$, one being holomorphic at $y_c$ and the other one having a "$y \log y$" branch at $y_c$.
More precisely, using classical methods from the theory of singular holomorphic ordinary differential equations,
 there exists a smooth solution $\psi_A$ of the form
$$
\psi_A(y) = (y - y_c) P_A(y - y_c),
$$
where $P_A$ is an holomorphic function in $y - y_c$, and another solution $\psi_B$ of the form
$$
\psi_B(y) = P_B(y - y_c) + {U''_c \over U_c'}  (y - y_c) P_A(y - y_c) \log(y - y_c) ,
$$
where $P_B$ is another holomorphic function in $y - y_c$. Note that both $P_A$ and $P_B$ depend in an holomorphic way on  $\alpha$ and $c$.
The existence of $P_A$ is obtained by searching it under the form
$$
P_A(y - y_c) = \sum_{n \ge 1} a_n (y - y_c)^n
$$
and inserting this expression in (\ref{Rayleighwithout}). This leads to a recurrence relation on the $\alpha_n$. It turns out that its radius of convergence
is positive. Then $P_A$ can be extended to a "pencil like" neighbourhood  $\Gamma_{\beta_0,\beta_1}$ of $\rit_+$ by analytic continuation.
Next $\psi_B$ is constructed through the method of the variation of parameters.
Direct computations \cite{Reid} show that one can choose
$$
P_A(y - y_c) = 1 + {U_c'' \over 2 U_c'} (y - y_c) + (U_c'' + \alpha^2) (y - y_c)^2 + ...
$$
and
$$
P_B(y - y_c) = 1 + \Bigl({U'''_c \over 2 U_c'} - {U_c'' \over U_c'^2} + {\alpha^2 \over 2} \Bigr) (y - y_c)  + ...,
$$
where $U_c' = U_s'(y_c)$ and similarly for $U_c''$ and $U_c'''$.
Note that the Wronskian of these two solutions is constant, and equals $-1$.

We have to carefully define the logarithm.
A natural choice is to define the logarithm except for negative values of its argument,
namely when $y - y_c$ is outside $- \rit_+$. 
However we are interested in $c$ and not $y_c$, and more precisely in $c$ with $\Im c >0$. In this
case $\Im y_c > 0$. We thus want to define $\log(y - y_c)$
when $\Im y_c > 0$. Thus we take the determination of the logarithm which is defined on $\cit - i \rit_+$.

We observe that, as $y$ goes to $+ \infty$, the only possible asymptotic behaviours of the solutions are $e^{+ \alpha y}$ and
$e^{- \alpha y}$, assuming $\alpha > 0$ to fix the ideas.
We therefore have to find a combination  $\psi_{-,\alpha}$ of
$\psi_A$ and $\psi_B$ which is equivalent to $e^{- \alpha y}$ at infinity.
For this we introduce, following J.W. Miles \cite{Miles} (see also \cite{Reid}, section $3.4$, page $277$),
\beq \label{definitionOmega}
\Omega(y) = {\psi \over (U_s - c) [ U_s' \psi - (U_s-c) \psi'] }
\eeq
and note that, after some calculations, $\Omega(y)$ satisfies the ordinary differential equation
\beq \label{equationOmega}
\Omega' = \alpha^2 Y \Omega^2 - Y^{-1}
\eeq
where $Y(y) = (U_s - c)^2$.  We note that the source term $Y^{-1}$ is singular when $y = y_c$.
By definition of $\Omega(y)$, for  $\psi = \psi_{-,\alpha}$, 
the corresponding function $\Omega(y)$ satisfies 
$$
\lim_{y \to + \infty} \Omega(y) = {1 \over \alpha ( U_+ - c)^2} .
$$
Using (\ref{equationOmega}), we expand $\Omega(y)$ in
$$
\Omega(y) = {1 \over \alpha (U_+ - c)^2} + \theta(y)
$$ 
where $\theta$ satisfies
\beq \label{equtheta}
\theta' = \Bigl[ {(U_s - c)^2 \over (U_+ - c)^4} - {1 \over (U_s - c)^2} \Bigr]
+ 2 \alpha {(U_s - c)^2 \over (U_+ - c)^2} \theta + \alpha^2 (U_s - c)^2 \theta^2 .
\eeq
Let
$$
\Omega_0(y,c) = -{1 \over ( U_+ - c)^2} \int_y^{+ \infty} \Bigl[
{(U_s(z) - c)^2 \over (U_+ - c)^2} - {(U_+ - c)^2 \over (U_s(z) - c)^2} \Bigr] dz . 
$$
Note that the integral involved in $\Omega_0(y,c)$ is well defined since $U_s(z)$ converges exponentially fast to $U_+$ at infinity.
Note also that the second term defining $\Omega_0(y,c)$  is singular at $y = y_c$.
More precisely,
\beq \label{exx}
{1 \over (U_s(z) - c)^2} = {1 \over U_c'^2 (z -y_c)^2} - {U_c'' \over U_c'^3} {1 \over z - y_c} + \cdots.
\eeq
Thus $\Omega_0(y,c)$ behaves like $(y - y_c)^{-1}$ near $y_c$. As a consequence, $(U_s - c) \Omega_0(y,c)$ is bounded.
Using (\ref{equtheta}) we then get
\beq \label{expansionOmega}
\Omega(0) = {1 \over \alpha (U_+ - c)^2} + \Omega_0(y,c) + O(\alpha).
\eeq
We now combine (\ref{expansionOmega}) with
\beq \label{Omegaat0}
\Omega(0) = -  {\psi(0) \over c [ U'_s(0) \psi(0) + c \psi'(0) ]} 
\eeq
and obtain
$$
\Omega(0)^{-1} = \alpha (U_+ - c)^2 - \alpha^2 (U_+ - c)^4 \Omega_0(0,c) + O(\alpha^3) 
$$
\beq \label{valueOmega}
=  - c U_s'(0) - c^2 {\psi_{-,\alpha}'(0) \over \psi_{-,\alpha}(0)},
\eeq
which leads to (\ref{psippsi}).

Let us go back to the question of writing $\psi_{-,\alpha}$ as a combination of $\psi_A$ and $\psi_B$. 
Let 
$$
r = {\psi_{-,\alpha}'(0) \over \psi_{-,\alpha}(0)} = O \Bigl({1 \over c} \Bigr) + O \Bigl( {\alpha \over c^2} \Bigr) .
$$ 
Let us look for $\psi_{-,\alpha}$ under the form
$$
\psi_{-,\alpha} = \psi_A + \theta \psi_B 
$$
for some $\theta$. We have
$$
\theta = { r \psi_A(0) - \psi_A'(0) \over \psi_B'(0) - r \psi_B(0)} .
$$
Then, using the explicit expressions of $\psi_A$ and $\psi_B$, we obtain $\psi_A(0) = O(c)$, $\psi_A'(0) = 1 + O(c)$,
$\psi_B(0) = 1 + O(c \log c)$ and $\psi_B'(0) = O(\log c)$.
Thus, 
$$
r \psi_A(0) - \psi_A'(0) =  O(r c)  - 1 + O(c)  = O(1) + O \Bigl( {\alpha \over c} \Bigr),
$$
and 
$$
\psi_B'(0) - r \psi_B(0) = O(\log c) - O(r) = O \Bigl({1 \over c} \Bigr) + O \Bigl( { \alpha \over c^2 } \Bigr).
$$ 
Thus,
\beq \label{sizetheta}
\theta = O(c).
\eeq
Later we will need the asymptotic behaviour of $\Omega_0(0)$ when $\Im c$ is small. Using (\ref{exx})
we note that, as $\Im c \to 0$, 
\beq \label{imomega0}
\Im \Omega_0(0,c) \to - i \pi {U_c'' \over U_c'^3} (U_+ -c)^2 
\eeq
and
\beq \label{imomega1}
\Re \Omega_0(0,c) \sim {(U_+ - c)^2 \over U_c'^2 y_c} + O(\log y_c),
\eeq
 which ends the proof.
\end{proof}


\subsection{Green function for Rayleigh's equation}

.

We now define the Green function $G_{\alpha,c}(x,y)$ to be the solution of
\beq \label{defiG0c}
Ray_{\alpha,c}(G_{\alpha,c}(x,y)) = \delta_x
\eeq
which satisfies $G_{\alpha,c}(x,0) = 0$ and $G_{\alpha,c}(x,y) \to 0$ as $y \to + \infty$. 
The solution of Rayleigh's equation
\beq \label{Ray0cf}
 (U_s - c)  (\partial_y^2 - \alpha^2 ) \psi - U_s''  \psi  = f,
 \eeq
 together with its two boundary conditions, at $0$ and $+ \infty$, is then explicitly given by
 \beq \label{Ray0cf1}
 \psi(y) = \int_0^{+\infty} G_{\alpha,c}(x,y) f(x) dx .
 \eeq
 We first solve (\ref{defiG0c}) without the boundary conditions, just choosing a particular solution.
 This leads to the introduction of
 \beq \label{GRay}
  G_{\alpha,c}^{int}(x,y) = {1 \over U_s(x) - c} 
 \Bigl\{ \begin{array}{c}
 \psi_{-,\alpha,c}(y) \psi_{+,\alpha,c}(x) \qquad \hbox{if} \qquad y > x \cr
 \psi_{-,\alpha,c}(x) \psi_{+,\alpha,c}(y) \qquad \hbox{if} \qquad y < x , \cr
\end{array} 
 \eeq
 where we used that the Wronskian between the two solutions equals $1$.
 We then correct $G_{\alpha,c}^{int}$ by $G_{\alpha,c}^b$ such that 
 \beq \label{GRay2}
 G_{\alpha,c} = G_{\alpha,c}^{int} + G_{\alpha,c}^b
 \eeq
 satisfies both Rayleigh's equation and the two boundary conditions. This leads to 
 \beq \label{GRay3}
 G_{\alpha,c}^b(x,y) = - {\psi_{-,\alpha,c}(x) \psi_{+,\alpha,c}(0)  \over U_s(x) -c}  {\psi_{-,\alpha,c}(y) \over \psi_{-,\alpha,c}(0)}.
 \eeq


\subsection{Proof of Theorem \ref{growthRayleigh}}


 We now focus on the inverse Laplace transform of the Green function, namely on Dunford's formula
 \beq \label{mainformula}
 G_\alpha(t,x,y) = - {\alpha \over 2  \pi} \int_\Gamma e^{- i \alpha c t} G_{\alpha,c}(x,y) \; dc 
 \eeq
 where $\Gamma$ is a contour  "above"  $\sigma(U_s)$.
 We can shift the contour $\Gamma$ very close to the range of $U_s$ and for instance choose
 $$
\Gamma = \Gamma_- \cup \Gamma_c \cup \Gamma_+
$$
where $\Gamma_c =  [\min(U_s) - \gamma , \max(U_s) + \gamma] + i \eps$, with $\eps > 0$ and $\gamma > 0$,
and where $\Gamma_-$ and $\Gamma_+$ are half lines going to infinity like
$\Gamma_- =  (\min(U_s) - \gamma) + i \eps + (1 + i) \rit_-$ and similarly for
$\Gamma_+ = ( \max(U_s) + \gamma) + i \eps + (-1 + i) \rit_+$.

We will now shift the contour $\Gamma$ downwards, namely from positive $\eps$ to negative $\eps$. Meanwhile we will
encounter poles and branches of the Green function $G_{\alpha,c}$.
We recall that $G_{\alpha,c}$ is meromorphic outside $U_s(x) -  i \rit_+$ and $U_s(y) - i \rit_+$, with a pole at $c = 0$ for $G_{\alpha,c}^b$.
We decompose $G(t,x,y)$ in $G_{\alpha,c}^{int}(t,x,y)$ and $G_{\alpha,c}^b(t,x,y)$.
Let us first focus on $G_{\alpha,c}^{int}$.

We only detail the case $y > x$, the case $y < x$ being similar. We have
$$
G_{\alpha,c}^{int}(t,x,y) = - {\alpha \over 2  \pi} \int_{\Gamma} e^{- i \alpha c t } 
{\psi_{-,\alpha,c}(y) \psi_{+,\alpha,c}(x) \over U_s(x) - c} \; dc  = G_1 + G_2 + G_3 + G_4
$$
where
$$
G_1 =  {1 \over 2 i \pi} \int_{\Gamma} e^{- i \alpha c t }  { P_{-,\alpha,c}(y - y_c) P_{+,\alpha,c}(x - y_c) \over U_s(x) - c} \; dc,
$$
$$
G_2 =  {1 \over 2 i \pi} \int_{\Gamma} e^{- i \alpha c t }  
{Q_{-,\alpha,c} (y - y_c) \log( y - y_c) P_{+,\alpha,c} \over U_s(x) - c} \; dc ,
$$
$$
G_3 =  {1 \over 2 i \pi} \int_{\Gamma} e^{- i \alpha c t }  
{P_{-,\alpha,c} (x - y_c) \log( x - y_c) Q_{+,\alpha,c} \over U_s(x) - c} \; dc . 
$$
and
$$
G_4 =  {1 \over 2 i \pi} \int_{\Gamma} e^{- i \alpha c t }  
{  (y - y_c) Q_{-,\alpha,c} \log( y - y_c)  (x - y_c) \log( x - y_c) Q_{+,\alpha,c} \over U_s(x) - c} \; dc . 
$$
We now move $\Gamma$ downwards to negative $\eps$ in $G_1$. We have to cross a singularity at $c = U_s(x)$, hence
$$
G_1 = e^{-i \alpha U_s(x) t} P_{-,\alpha,U_s(x)}(y-x) P_{+,\alpha,U_s(x)}(0) + O(e^{-\eps t}). 
$$
Let us turn to $G_3$. We note that $(x - y_c) / (U_s(x) - c)$ is smooth, so
$G_3$ only involves a logarithmic branch and may be rewritten
$$
G_3 =  {1 \over 2 i \pi} \int_{\Gamma} e^{- i \alpha c t } R_c(x - y_c)  \log( x - y_c) \; dc  
$$
where $R_c$ is analytic. We move $\Gamma$ downwards, from $+ i \eps$ to $- i \eps$. Doing this
we are left with a branch where the logarithm is not defined.  Turning around this logarithmic branch, we see that its contribution leads to
$$
G_3 = O( \langle \alpha t \rangle^{-1}).
$$
The study of $G_2$, $G_4$ and $G_{\alpha,c}^b$ is similar, which ends the proof of Theorem \ref{growthRayleigh}.


\subsection{Inversion of Rayleigh operator  \label{further}}


Let us first introduce a norm on the holomorphic functions defined in the vicinity of $0$.
Let $\rho > 0$.
If 
$$
\phi(Y) = \sum_{n \ge 0} a_n Y^n
$$
then we define its norm $\| \phi \|_\rho$ by
$$
\| \phi \|_\rho = \sum_{n \ge 0} | a_n | \rho^n  .
$$
We have the following result.

\begin{proposition} \label{solutionRayleigh}
Let $P_1$ and $Q_1$ be two holomorphic functions in a neighbourhood of $y_c$.  Then there exists a solution $\phi$ of
\beq \label{ee1}
Ray_{\alpha, c}(\phi) = P_1(y - y_c) + (y - y_c) \log(y - y_c) Q_1(y - y_c)
\eeq
which is of the form 
\beq \label{ee2}
\phi = P(y - y_c) + (y - y_c) \log(y - y_c) Q(y - y_c) 
\eeq
in a neighborhood of $y_c$. Moreover, provided $\rho$ is small enough (independently of $P$ and $Q$),
\beq \label{ee3}
\| P \|_\rho + \| Q \|_\rho \lesssim \| P_1 \|_\rho + \| Q_1 \|_\rho .
\eeq
\end{proposition}

\Remark
Of course the solution is not unique since we may add any linear combination of $\psi_{\pm,\alpha,c}$.

\begin{proof}
We insert (\ref{ee2}) in (\ref{ee1}) and identify the logarithmic terms, which gives
\beq \label{eqP}
(U_s -  c) Q'' + 2 {U_s -  c \over y - y_c} Q' =  \Bigl[ U_s'' + (U_s -  c) \alpha^2 \Bigr] Q + Q_1 
\eeq
and
\beq \label{eqQ}
(U_s -  c) ( P'' - \alpha^2 P + 2 Q') + {U_s -  c  \over  y - y_c} Q = U_s'' P + P_1.
\eeq
Let us locally solve (\ref{eqP}) and (\ref{eqQ}) by looking for solutions under the form 
$$
Q(y) = \sum_n a_n (y - y_c)^n, \qquad P(y) = \sum_n b_n (y - y_c)^n .
$$
Let us expand $Q_1$ and $P_1$ using similar series, with $a_n^1$ and $b_n^1$ instead of $a_n$ and $b_n$.
 Let $Y = y - y_c$. We first rewrite (\ref{eqP}), under the form
 $$
 Y Q'' + 2 Q' = \phi_1 Q + \phi_2 Q_1
 $$
 where 
 $$
 \phi_1 = {Y \over U_s - c}  \Bigl[ U_s'' + (U_s -  c) \alpha^2 \Bigr], \qquad \phi_2 = {Y \over U_s - c} 
 $$ 
 are holomorphic functions of $Y$. Let $\sum_n c_n^1 Y^n$ and $\sum_n c_n^2 Y^n$ be their expansions near $Y = 0$.
 There is no uniqueness of the solution $Q$, so we may impose for instance $Q(y_c) = 0$, namely $a_0 = 0$.
 This leads to $a_1 =  c_0^2 a_0^1/2$, and to the induction relation
 \beq \label{relation}
\Bigl[ (n+1) n + 2n+2 \Bigr] a_{n+1} = \sum_{j = 0}^n c_{n-j}^1  a_{j} + \sum_{j=0}^n c_{n-j}^2 a_j^1
 \eeq
for $n \ge 1$.
Let $\rho$ such that $\| \phi_1 \|_{\rho} < + \infty$, $\| \phi_2 \|_{\rho} < + \infty$ and $\| Q_1 \|_{\rho} < + \infty$.
Using (\ref{relation}), we obtain for $n \ge 1$,
$$
 | a_{n+1} | \rho^{n+1} \le {\rho \over n^2}  \Bigl( \sum_{j = 0}^n | c_{n-j}^1 |  | a_{j} | \rho^n + \sum_{j=0}^n | c_{n-j}^2 |  | a_j^1 | \rho^n \Bigr),
 $$
 thus
$$
\sum_{j=0}^{n+1} | a_j | \rho^j \le  {1 \over 2} | c_0^2 a_0^1 | \rho 
+ \rho   \Bigl( \sum_{j +k = 0}^n | c_{j}^1 |  | a_k | \rho^{j+k} + \sum_{j +k =0}^n | c_j^2 |  | a_k^1 | \rho^{j+k} \Bigr)
$$
$$
 \le  {1 \over 2} | c_0^2 a_0^1 | \rho 
+ \rho   \Bigl( \sum_{j =0}^n | c_j^1 | \rho^j \sum_{j=0}^n | a_j | \rho^j
+ \sum_{j= 0}^n | c_{j}^2 | \rho^j  \sum_{j=0}^n  | a_j^1 | \rho^{j}  \Bigr).
$$
If $\rho \| \phi_1 \|_{\rho} < 1/2$ then
$$
\sum_{j=0}^{n} | a_j | \rho^j \le    | c_0^2 a_0^1 | \rho  +  2 \rho  \sum_{j= 0}^n | c_{j}^2 | \rho^j  \sum_{j=0}^n  | a_j^1 | \rho^{j}   
\lesssim \| Q_1 \|_\rho,
$$
and thus
$$
\| Q \|_\rho \lesssim \| Q_1 \|_\rho.
$$
Equation (\ref{eqQ}) can be treated similarly.
\end{proof}

We now want to solve Rayleigh equation globally. For this we introduce $Y_{\rho,\sigma}$ to be the space of functions $f$ 
defined on the pencil domain $\Gamma_{\beta_0,\beta_1}(y_c)$, such that
$$
\| f(x) e^{\sigma x} \|_{L^\infty(\Gamma_{\beta_0,\beta_1}(y_c))} < + \infty
$$
and such that, near $y_c$, $f$ is of the form $f = P + (y - y_c) \log(y - y_c) Q$, with $\| P \|_\rho + \| Q \|_\rho < + \infty$.
We define $\| f \|_{Y_{\rho,\sigma}}$ as the sum of these three norms
$$
\| f \|_{Y_{\rho,\sigma}} = \| f \|_{X_\sigma} + \| P \|_\rho + \| Q \|_\rho.
$$
We now prove the following Proposition

\begin{proposition} \label{propXrho}
Let $\sigma > \alpha$.
If $\rho$ is small enough, then for any $f \in Y_{\rho,\sigma}$, there exists $\phi \in Y_{\rho,\alpha}$ such that
$$
Ray_{\alpha,c} \phi = f .
$$
Moreover,
\beq \label{dee1}
\| \phi \|_{Y_{\rho,\alpha}} \lesssim \| f \|_{Y_{\rho,\sigma}} 
\eeq
and, provided $\sigma < \beta$,
$$
\| (U_s - c) (\partial_y^2 \phi - \alpha^2 \phi ) \|_{Y_{\rho,\sigma}}  \lesssim \| f \|_{Y_{\rho,\sigma}}.
$$
\end{proposition}

\Remarks
Note that, as previously, the vorticity $(\partial_y^2 - \alpha^2) \phi$ decays faster than the stream function $\phi$ which decays slowly,
like $e^{- | \alpha | y}$.

\begin{proof}
For $y$ near $y_c$, we use the previous Proposition, which gives some function $\phi_0$ for $| y - y_c | \le \rho$.
For $\Re y > \Re y_c + \rho/2$, we use the Green function $G^{int}_{\alpha,c}$ to extend $\phi_0$
through
$$
\phi_0(y) = \int_{\rho/2}^{+\infty} G^{int}_{\alpha,c}(x,y) f(x) \, dx + A \psi_{+,\alpha,c}(y) + B \psi_{-,\alpha,c}(y),
$$
where the constants $A$ and $B$ are adjusted such that $\phi_0$ is $C^1$ at $\Re y_c + \sigma /2$.
Note that this  formula  can be extended to $\Gamma_{\beta_0,\beta_1}$.
Now at $\phi_0$ behaves like $A e^{+ | \alpha | y}$ at infinity, the desired solution  is $\phi = \phi_0 - A e^{+ | \alpha | y}$.

Using the explicit expression of $G^{int}_{\alpha,c}$ we get (\ref{dee1}).
Next the vorticity $\partial_y^2 \phi - \alpha^2 \phi$ satisfies
$$
(U_s - c) (\partial_y^2 \phi - \alpha^2 \phi ) = U_s'' \phi + f
$$
As $|U_s''(y)| \lesssim e^{-\beta y}$, the right hand side decays like $e^{-\sigma y}$,
which ends the proof.
\end{proof}


\subsection{Further results \label{adjoint}}


The adjoint of Rayleigh's operator is given by 
\beq \label{adjointRay}
Ray_{\alpha,c}^t(\psi) = (\partial_y^2 - \alpha^2) (U_s - \bar c) - U_s'' \psi,
\eeq
and is conjugated with Rayleigh operator through
\beq \label{conjugation}
Ray_{\alpha,c}^t(\psi) = { Ray_{\alpha,\bar c}\Bigl( (U_s - \bar c) \psi \Bigr) \over U_s - \bar c}.
\eeq
Thus, two independent solutions $\psi^t_{\pm,\alpha,c}$ are obtained starting from  
$\psi_{\pm,\alpha,c}$  thanks to the formula
\beq \label{psitmodes1}
\psi^t_{\pm,\alpha,c} = {\psi_{\pm,\alpha,\bar c} \over U_s(y) - \bar c }.
\eeq
In particular, $\psi^t_{+,\alpha,c}$ is more singular than $\psi_{+,\alpha,c}$ at $y_c$, with a  $(U_s(y) - \bar c)^{-1}$ singularity.
On the contrary, $\psi^t_{-,\alpha,c}$ is close to $1$ for small $\alpha$.

Moreover, when $c$ is away from the range of $U_s$,  Rayleigh equation may be rewritten as
\beq \label{RRay}
\partial_y^2 \psi - \Bigl[ \alpha^2 + {U_s'' \over U_s - c} \Bigr] \psi = 0,
\eeq
where the quantity between brackets is bounded. It is then easy to construct two solutions $\psi_{\pm,\alpha,c}$. 
These solutions behave like $e^{\pm \alpha y}$ at infinity,
and all their norms remain bounded as long as $|c|$ remains bounded and away from the range of $U_s$.

As $|c|$ goes to $+\infty$, (\ref{RRay}) is a perturbation of $\partial_y^2 \psi - \alpha^2 \psi = 0$, and it is possible to write a full expansion 
of $\psi_{\pm,\alpha,c}$ in terms of $c^{-1}$. In particular, $\psi_{-,\alpha,c}$ is bounded as $|c|$ goes to infinity. We do not detail these points any further.


\section{Airy and construction of "fast" solutions}


In this section we first recall the solutions of linearized Airy equation and the classical Langer's transformation. 
We then construct independent solutions and  Green functions for $Airy$ and  $Airy \, \circ \, \partial_y^2$.
This allows us to construct solutions of Orr Sommerfeld equations with  rapidly decaying sources,
as well as two independent fast solutions $\phi_{f,\pm}$ of this equation without forcing term.


\subsection{Classical Airy equation}


We recall that the solutions of the classical Airy equation
$$
  \partial_y^2 \psi = y \psi
 $$
 are linear combination of the classical Airy functions $Ai$ and $Bi$, whose properties are gathered in the Appendix \ref{appendix1}.

We now turn to the study of the linearized Airy equation
\beq \label{linearAiry}
 \eps \partial_y^2 \psi = U_s'(y_c) (y - y_c) \psi.
\eeq
Note that $\eps$ is purely imaginary, and that $y_c$ is a complex number with positive real part.
Let $\gamma$ be defined by
$$
\eps \gamma^3 = U_s'(y_c) ,
$$
namely by
\beq \label{defigamma}
\gamma = \Bigl(  {i \alpha U_s'(y_c) \over \nu} \Bigr)^{1/3} .
\eeq
We note that $\gamma$ is a complex number and that its argument equals
$$
\arg(\gamma) = {\pi \over 6} + {\arg(U_s'(y_c)) \over 3}.
$$
As $|\arg(y_c) | \le \pi/2$, we have
$$ 
0 < \arg(\gamma) < {\pi \over 3} .
$$
Now, $Ai(\gamma (y - y_c))$ and $Ci(\gamma (y - y_c))$, defined in (\ref{defiCi}), 
are two independent solutions of (\ref{linearAiry}) with Wronskian of order $\gamma$.
  A short computation shows that the first solution
goes to $0$ as $y \to + \infty$, that the second one diverges as $y \to + \infty$, and conversely when $y \to - \infty$.

Let us study how the solutions of (\ref{linearAiry}) scale with $\eps$.
Let us consider
\beq \label{la}
\eps \partial_y^2 \phi - U_s'(y_c) (y - y_c) \phi = \delta_x.
\eeq
The Green function $G^{Airy}(x,y)$ of (\ref{la}) is explicitly given by
\beq \label{GAiry}
G^{Airy}(x,y) = {1 \over \eps \gamma W^{Airy}}  \left\{ \begin{aligned} 
Ci(\gamma (y - y_c) ) Ai(\gamma (x - y_c) ) \quad \hbox{if} \quad y < x,
\\
 Ci(\gamma (x - y_c) ) Ai(\gamma (y - y_c) ) \quad \hbox{if} \quad y > x,
 \end{aligned}\right.
\eeq
where $W^{Airy}$ is the Wronskian between $Ai$ and $Ci$, which is of order $1$.

Let us detail some properties of $G^{Airy}(x,y)$. First, $G^{Airy}(x,y)$ is symmetric in $x$ and $y$.
Next, using (\ref{asymptAi}) and (\ref{asymptAi1}), provided $\gamma (x - y_c)$ is large, we have
$$
\int_x^{+ \infty} | G^{Airy}(x,y) | \, dy \lesssim {1 \over  |x - y_c|}.
$$
 If $\gamma (x - y_c)$ is small, then the same integral is directly bounded by $\eps^{-1} \gamma^{-2}$ since $Ai \in L^1$.
Thus,
\beq \label{studyG}
 \int_x^{+\infty} | G^{Airy}(x,y) | \, dy \lesssim  \min \Bigl(  |x - y_c|^{-1}, \gamma \Bigr).
\eeq
The computations are similar for the integral between $- \infty$ and $0$.

We note the behavior of this integral in $(x - y_c)^{-1}$, when $x - y_c$ is large, which is coherent with the fact that, 
when $\eps = 0$, Airy operator degenerates into $-(y - y_c) \phi$.
When $x -y_c$ is small, the singularity of $(x - y_c)^{-1}$ is smoothed out, and saturates at $\gamma$.


\subsection{Langer's transformation}


To go from linearized Airy equation (\ref{linearAiry}) to $Airy$, we  need to use the classical Langer's transformation \cite{Drazin} that we recall now.

\begin{lemma} Langer's transformation.\\
Let $B(y)$ and $C(y)$ be two functions. Let $f$ and $g$ such that
\beq \label{Langer1}
B \Bigl( g(y) \Bigr)  (g')^2 = C(y)
\eeq
and
\beq \label{Langer2}
2 f' g' + f g'' = 0 .
\eeq
If $\psi$ is a solution to
\beq \label{Langer3}
- \eps \psi'' + B(y) \psi = 0,
\eeq
then
$$
\phi(y) = f(y) \psi \Bigl( g(y) \Bigr)
$$
is a solution to
\beq \label{Langer4}
- \eps \phi'' + C(y) \phi =- \eps {f'' \over f} \phi . 
\eeq
\end{lemma}

\begin{proof}
The proof relies on the following explicit computation
$$
- \eps \phi'' + C(y) \phi = - \eps f'' \psi -  2 \eps f' \psi' g' - B(g(y))  (g')^2 f \psi - \eps f \psi' g'' +  C(y) f\psi.
$$
Note that $f$ may be seen as a modulation of amplitude and $g$ as a change of phase.
\end{proof}

We now detail Langer's computations and construct $f(y)$ and $g(y)$ in the particular case 
$$
B(y) =  U_s'(y_c) (y - y_c), \qquad C(y) =  U_s(y) - c .
$$
Using (\ref{Langer1}), we get
$$
g'^2(y) = {U_s(y) - c \over U_s'(y_c) (g(y) - c)} .
$$
We thus face two difficulties. First, we must take the square root of the right hand side, and second, the right hand side is singular when $g(y) = c$.

We may assume that $\Re U_s'(y_c) > 0$, provided $c$ is small enough. 
Let us choose for $\sqrt{y - y_c}$ the usual square root which is defined provided $y - y_c \notin \rit_-^\star$, where its value changes 
by a  multiplicative factor $-1$.
We then define $B_1$ to be the primitive of $\sqrt{B}$. Then $B_1$ is well defined and smooth except when $y - y_c \notin \rit_-^\star$,
where its value change by a multiplicative factor $-1$.

Similarly we define $\sqrt{C}$ to be the usual square root of $U_s(y) - c$, which is well defined except when $U_s(y) - c \in \rit_-^\star$.
Let $C_1$ be its primitive, defined except when $U_s(y) - c \in \rit_-^\star$, where it has a multiplicative factor of $-1$.
 
Using (\ref{Langer1}),  we obtain
$$
B_1(g(y)) = C_1(y) .
$$
This defines $g(y)$ in a smooth way, except at points where $C_1(y)$ is not smooth, namely when $U_s(y) - c \in \rit_-$.
When we cross this line, $C_1(y)$ is multiplied by $-1$. But $B_1(g(y))$ is also multiplied by the same factor if $g(y)$ crosses $y - y_c \in \rit_-^\star$
at the same time. Thus $g(y)$ is smooth, and the choice of the square root is not important.
This gives
\beq \label{explicitg}
g(y) = y_c + \Bigl( {3 \over 2 \sqrt{U_s'(y_c)}} \int_{y_c}^y  \sqrt{U_s(z) - c} \, dz \Bigr)^{2/3}.
\eeq 
Using the explicit expression of $B_1$, we see that $g(y)$ is well defined for any $y$ in a neighborhood of $\rit_+$.
Moreover, $g(y_c) = y_c$ and $g'(y_c) = 1$. Note that $g(y)$ is in general a complex number since $c$ is a complex number.

When $y \to + \infty$, we have
\beq \label{asymptg}
g(y) \sim \Bigl( {3 \over 2} \Bigr)^{2/3} \Bigl( {U_+ - c \over U_c'} \Bigr)^{1/3} y^{2/3},
\eeq
and $g'(y)$ behaves like $y^{-1/3}$.
Now we note that (\ref{Langer2}) can be solved by choosing
\beq \label{defiff}
f(y) = {1 \over \sqrt{g'(y)}} ,
\eeq
which can be made explicit using (\ref{explicitg}).
For large $y$, $f(y)$ behaves like $y^{1/6}$.
As a consequence, $f'' /f$ is uniformly bounded  in a neighborhood of $\rit_+$
and decays like $y^{-2}$ at infinity.

Langer's transformation leads us to introduce the modified Airy functions
$$
Ai_a(y) = {1 \over \sqrt{g'(y)}} Ai \Bigl( \gamma (g(y) - y_c) \Bigr)
$$
and
$$
Ci_a(y) = {1 \over \sqrt{g'(y)}} Ci \Bigl( \gamma (g(y) - y_c) \Bigr).
$$
As $f'' / f$ is uniformly bounded, these two functions $Ai_a$ and $Ci_a$ are approximate solutions of Airy equation in the sense that
\beq \label{approxAiry}
| Airy(Ai_a) | \lesssim \eps |Ai_a|,  \qquad | Airy(Ci_a) | \lesssim \eps |Ci_a|. 
\eeq
Let us briefly describe $Ai_a$.
As $y \to + \infty$, combining (\ref{asymptg}) and (\ref{asymptAi}), we get
\beq \label{largeAiy}
\log Ai_a(y) \sim -C_a \gamma^{3/2} y 
\eeq
for some positive constant $C_a$, and similarly, $\log Ci_a \sim C_c \gamma^{3/2} y$.
Thus, for large $y$, both $Ai_a$ and $Ci_a$ have an exponential behavior with a speed $\gamma^{3/2}$, larger than $\gamma$.
This is coherent with the observation that, for large $y$, solutions to $Airy \, \phi = 0$ behave like solutions to $\eps \partial_y^2 \phi = U_+ \phi$,
namely like exponentials of speed $\sqrt{\eps^{-1} U_+}$, which is of order $\eps^{-1/2}$, namely of order $\gamma^{3/2}$.

For small $\gamma (y - y_c)$ however, the natural scale which appears in $\gamma^{-1}$. Let us detail this point and introduce
the logarithmic derivative
\beq \label{definitionmu}
\mu(y) = {\partial_y Ai_a(y) \over Ai_a(y)},
\eeq
which can be seen as a measure of the "scale" of variations of $Ai_a$. Then, provided $\gamma (y - y_c)$ is large,
$$
\mu(y) =  \gamma g'(y)  {Ai'(\gamma (g(y) - y_c)) \over Ai( \gamma (g(y) - y_c))} 
\Bigl[ 1 + O(\gamma^{-1}) \Bigr]
$$
$$
\sim  \gamma^{3/2} g'(y) (g(y) - y_c)^{1/2}  \Bigl[ 1 + O(\gamma^{-1}) \Bigr].
$$
Thus $\mu(y) \to C \gamma^{3/2}$ when $\gamma (y - y_c) \to + \infty$. On the contrary, for bounded $\gamma (y - y_c)$,
as $g'(y_c) = 1$,
$$
\mu(y) =\gamma g'(y)^2   \Bigl[ \gamma (y - y_c) \Bigr]^{1/2} \Bigl[ 1 + o(1) \Bigr],
$$
thus $\mu(y)$ is of order $\gamma$.
We thus have a transition between a decay rate $\gamma$ and $\gamma^{3/2}$ as $y - y_c$ increases.

We also note that $g(y)$ is defined in a neighborhood of $\rit_+$ and thus is well defined for $| \Im y | \le  \min( \beta_0 \Re y, \beta_1)$,
provided $\beta_1$ is small enough. As a consequence, $Ai_a$ is defined on a neighborhood of $\rit_+$.
However it is exponentially increasing when $\Im y$ increases, and only remains bounded when $| \Im y | \lesssim | \gamma |^{-1}$.
This remark will be crucial when we will move the integration contour  of the Green function for Navier Stokes equations downwards into negative
$\Im c$.  We must take
\beq \label{constraintdelta}
\beta_1 \lesssim | \gamma |^{-1}.
\eeq
In particular, the width of the "pencil domain" $\Gamma_{\beta_0,\beta_1}$ goes to $0$ as $\nu$ goes to $0$.


\subsection{Study of $Airy$}


We now construct two independent solutions to $Airy$ equation,  starting from the approximate solutions $Ai_a$ and $Ci_a$. 
For this we will design an iterative scheme, using an approximate solver for $Airy$ equation.

More precisely we first introduce the approximate Green function
\beq \label{Gaiapp}
G^{Ai}_{app}(x,y) = {1 \over \eps W^{a}}  \left\{ \begin{aligned} 
Ci_a(y) Ai_a(x) \quad \hbox{if} \quad y < x,
\\
 Ci_a(x) Ai_a(y) \quad \hbox{if} \quad y > x,
 \end{aligned}\right.
\eeq
where $W^{a}$ is the Wronskian determinant of $Ai_a$ and $Ci_a$, which is of order $\gamma$.
Note that $G^{Ai}_{app}$ is symmetric in $x$ and $y$.
Let us detail some properties of $G^{Ai}_{app}$.

\begin{lemma}
Let us define $\zeta(x)$ by
\beq \label{defizeta}
\zeta(x) = \min \Bigl( |\gamma|, \max(|x -y_c|^{-1},1) \Bigr)
\eeq
and $\theta(x)$ by 
$$
\theta(x) = {2 \over 3} \Re \Bigl[ \gamma^{3/2} (g(x) - y_c)^{3/2} \Bigr].
$$
Then we have
\beq \label{boundintG}
\int | G^{Ai}_{app}(x,y) | \, dy \lesssim \zeta(x).
\eeq
and, for $\mu \le \gamma / 10$,
\beq \label{boundintG2}
\int | G^{Ai}_{app}(x,y) | e^{- \mu y} \, dy \lesssim e^{-\mu x}  \zeta(x) .
\eeq
Moreover, for $z > x$,
\beq \label{boundintG3}
\int_{y \ge z} \ | G^{Ai}_{app}(x,y) |  \, dy \lesssim  \zeta^{1/4}(x) \zeta^{3/4}(z) e^{\theta(x) - \theta(z)} .
\eeq
\end{lemma}

\Remarks
For large $y$, $U_s(y)$ is bounded, so as $x \to + \infty$, Airy equation turns to
$\eps \partial_y^2 \phi = U_+ \phi + \delta_x$, whose solutions are exponentials, in coherence with (\ref{largeAiy}).
This leads to the maximum $\max(|x - y_c|^{-1},1)$ which appears in (\ref{boundintG}), 
and which was absent from the similar formula (\ref{studyG}) for $G^{Airy}$.
Note that the function $\zeta(x)$ describes three regimes,  namely $|x - y_c| \le |\gamma|^{-1}$, $|\gamma|^{-1} \le | x - y_c | \le 1$
and  $| x - y_c | \ge 1$, corresponding to three different regimes of Airy equation.

Moreover, as $G^{Ai}_{app}(x,y)$ is symmetric in $x$ and $y$, integrals over $x$ of $G^{Ai}_{app}(x,y)$ satisfy similar bounds.

\begin{proof}
We note that
$$
I := \int_{y \ge x} | G^{Ai}_{app}(x,y) | \, dy 
=  {| Ci_a(x) | \over | \eps W^a |}  \int_{y \ge x} {1 \over |\sqrt{g'(y)}|} \Bigl| Ai \Bigl( \gamma (g(y) - y_c) \Bigr) \Bigr| \, dy.
$$
Setting $z = g(y)$ we get
$$
I =    {| Ci_a(x) | \over | \eps W^a | }  \int_{z \ge g(x)} {1 \over |g'(y)|^{3/2}} \Bigl| Ai \Bigl( \gamma (z- y_c) \Bigr) \Bigr| \, dz.
$$
Using (\ref{asymptg}), we obtain that $g(y) \sim C y^{2/3}$ and $g'(y) \sim C y^{-1/3}$. Thus $(g')^{-3/2} \sim C y^{1/2} \sim C g(y)^{3/4} \sim C z^{3/4}$,
for various constants $C$. This gives
$$
| I |  \lesssim    {|Ci_a(x)| \over | \eps W^a |}  \int_{z \ge g(x)} \langle z  \rangle^{3/4}  \Bigl| Ai \Bigl( \gamma (z- y_c) \Bigr) \Bigr| \, dz,
$$
and, using (\ref{asymptAi}), provided $\gamma (x - y_c)$ is large enough,
$$
| I | \lesssim {1 \over | \eps W^a |} {1 \over | \sqrt{g'(x)} |} {e^{+ 2 \gamma^{3/2} (g(x) - y_c)^{3/2}/3} \over \gamma^{1/4} ( g(x) - y_c )^{1/4}}  
$$
$$
\times
\int_{z \ge g(x)} { \langle z  \rangle^{3/4} \over \gamma^{1/4} (z - y_c)^{1/4}} \exp \Bigl( {- 2 \gamma^{3/2} (z - y_c)^{3/2} \over 3 } \Bigr)\, dz
$$
$$
\lesssim {1 \over \eps W^a \gamma^2} {\langle g(x) \rangle \over g(x) - y_c}
\lesssim \min( |x - y_c|^{-1}, 1).
$$
If $\gamma (x - y_c)$ is small or bounded, then, as $Ai_a$ is integrable, we directly have $| I | \lesssim | \gamma |$.
The bound for $y \le x$ is similar, which gives (\ref{Gaiapp}).
The proof of (\ref{boundintG2}) is similar, up to an extra $e^{\mu y}$ factor in all the integrals, using the results recalled in Appendix \ref{appendix3}.

To prove (\ref{boundintG3}), we observe that the corresponding integral is bounded by
$$
| I |  \lesssim {1 \over | \eps W^a |} {1 \over | \sqrt{g'(x)} |} {e^{+ 2 \gamma^{3/2} (g(x) - y_c)^{3/2}/3} \over \gamma^{1/4} ( g(x) - y_c )^{1/4}}  
$$
$$
\times
\int_{t \ge g(z)} { \langle t  \rangle^{3/4} \over \gamma^{1/4} (t - y_c)^{1/4}} \exp \Bigl( {- 2 \gamma^{3/2} (t - y_c)^{3/2} \over 3 } \Bigr)\, dz
$$
$$
\lesssim
{ \langle g(x) \rangle^{1/4}  \langle g(z)  \rangle^{3/4} \over  ( g(x) - y_c )^{1/4} (g(z) - y_c)^{3/4}} 
e^{+ 2 \gamma^{3/2} (g(x) - y_c)^{3/2}/3 -  2 \gamma^{3/2} (g(z) - y_c)^{3/2}/3}
$$
Then provided $\gamma (x - y_c)$ and $\gamma (z - y_c)$ are large enough,
$$
| I | \lesssim \min( | x - y_c |^{-1},1)^{1/4} \min( | z - y_c |^{-1},1)^{3/4} e^{\theta(x) - \theta(z)} .
$$
The case $\gamma (z - y_c)$ bounded is straightforward.
\end{proof}

\begin{proposition} \label{Airy0}
There exists two independent solutions $\phi^{Ai}_\pm(y)$ of
$$
Airy(\phi^{Ai}_\pm) = 0
$$
such that
\beq \label{propphiAi}
\phi^{Ai}_-(y) = Ai_a(y) \Bigl[ 1 + O(\gamma^{-2}) \Bigr], 
\eeq
and
\beq \label{propphiCi}
\phi^{Ai}_+(y) = Ci_a(y) \Bigl[ 1 + O(\gamma^{-2}) \Bigr] . 
 \eeq
\end{proposition}

\begin{proof}
Using this approximate Green function, we can construct an approximate solver for the $Airy$ equation, defined by 
\beq \label{solverAi}
{\cal S}^{Ai}_{app} (\phi) (y) =  \int_{\rit_+} G^{Ai}_{app}(x,y) \phi(x) \, dy .
\eeq
Note that the integral, which holds on $\rit_+$ can be transformed using analyticity into any line going from $0$ to $+ \infty$ in the "pencil" domain 
$\Gamma_{\beta_0,\beta_1}$.

Following Appendix \ref{appendix2} (second example), we define the iterative scheme
$$
R_n = Airy \, ( {\cal S}^{Ai}_{app}  \phi_n ) = - \eps f^{-1} f''  {\cal S}^{Ai}_{app}  \phi_n , \qquad \phi_{n+1} = ({\cal S}^{Ai}_{app})^{-1}  R_n,
$$
starting with $\phi_0 = Ai_a$, or $\phi_0 = Ci_a$. As $f'' / f$ is bounded and decays like $y^{-2}$, 
we note that $Airy \, \circ \, {\cal S}^{Ai}_{app}$ is a contraction in $L^\infty(|Ai_a(x)| dx)$.
The series $\sum_n \phi_n$ thus converges, towards a solution $\phi^{Ai}_-$ of Airy equation. Moreover this series
is an asymptotic expansion of $\phi^{Ai}_-$ in powers of $\gamma \eps \sim \gamma^{-2}$, leading to (\ref{propphiAi}) and (\ref{propphiCi}).  
\end{proof}

Once genuine solutions $\phi^{Ai}_\pm$ of Airy equation are constructed, it is possible to build a genuine Green function $G^{Ai}$
 and a genuine solver ${\cal S}^{Ai}$ for this equation. The expression of $G^{Ai}$ is the same as that of $G^{Ai}_{app}$ given in (\ref{Gaiapp}),
 except that $Ai_a$ must be replaced by $\phi^{Ai}_-$ and $Ci_a$ by $\phi^{Ai}_+$.
 Similarly,  ${\cal S}^{Ai}$ is given by (\ref{solverAi}) up to the replacement of $G^{Ai}_{app}$ by $G^{Ai}$.
 The solution of
 $$
 Airy \, \phi = f
 $$
 is then simply given by 
 $$
\phi = {\cal S}^{Ai} f ,
 $$
 namely by convolution with the exact Green function $G^{Ai}(x,y)$.
 
 Let us now detail the bounds on the inverse of $Airy$.
 
 \begin{lemma} \label{solutionAiry}
 Let $\sigma > 0$ with $\sigma < \gamma / 10$.
 Let $f \in X_\sigma$ be an holomorphic function on $\Gamma_{\beta_0,\beta_1}$.
 Then there exists an holomorphic function $\phi \in X_\sigma$, defined on $\Gamma_{\beta_0,\beta_1}$,
 such that
 \beq \label{eqq23}
 Airy \, \phi = f,
 \eeq
 and such that
 \beq \label{bbb1}
 | \phi(y) | \lesssim \zeta(y) e^{-\sigma y} \| f \|_{X_\sigma} \lesssim | \gamma |  e^{-\sigma y} \| f \|_{X_\sigma},
 \eeq
\beq \label{bbb2}
 | \partial_y \phi(y) | \lesssim | \gamma |^2 e^{-\sigma y} \| f \|_{X_\sigma},
 \eeq
 and
 \beq \label{bbb3}
 | \partial_y^2 \phi(y) | \lesssim | \gamma |^3 e^{-\sigma y}  \| f \|_{X_\sigma}.
 \eeq
 \end{lemma}
 
 \Remarks
 As we do not enforce any boundary condition at $y = 0$, the solution $\phi$ to (\ref{eqq23}) is not unique.
 Estimates (\ref{bbb2}) and (\ref{bbb3}) may be seen as the quantification of the "smoothning effect" of $Airy^{-1}$.

 \begin{proof}
  We solve (\ref{eqq23}) by iteration, first introducing
$$
\phi_1(y) = \int_0^{+ \infty} G^{Ai}_{app} (x,y) f(x) \, dx.
$$
Using (\ref{boundintG2}), with the integral over $x$ instead of $y$, we obtain
$$
| \phi_1(y) | \le \zeta(y) e^{-\sigma y}  \| f \|_{X_\sigma}.
$$
Now 
$$
\eps \partial_y^2 \phi_1 = f + (U_s - y) \phi_1.
$$
We note that $\zeta(y) (U_s - y)$ is bounded, thus
$$
|\eps| | \partial_y^2 \phi_1 | \lesssim e^{-\sigma y}  \| f \|_{X_\sigma},
$$
which gives (\ref{bbb3}).
Now 
$$
Airy \, \phi_1 = - \eps {f'' \over f} \phi_1 = O(\gamma^{-2}),
$$
since $\eps \zeta(y) \le \eps \gamma \lesssim \gamma^{-2}$. The construction may then be iterated to give the Lemma.
Next, (\ref{bbb2}) comes by interpolation.
 \end{proof}

 
 \subsection{Study of $Airy \circ \partial_y^2$}
 
 
 We now study $Airy \circ \partial_y^2$, which leads to integrate twice the solutions of $Airy$ equation.
 We first build two independent solutions of $Airy \circ \partial_y^2$.
 
 \begin{lemma}
 There exists two solution $\phi^{Ai,2}_\pm$ to $Airy \, \circ \, \partial_y^2 = 0$, one decaying to zero when $y \to + \infty$ and the other
 when $y \to - \infty$. Moreover,
 \beq \label{exphiAi2}
\phi^{Ai,2}_-(y) = g'(y)^{-5/2} Ai \Bigl( \gamma (g(y) - y_c) , 2 \Bigr) \Bigl[ 1 + O(\gamma^{-2}) \Bigr] ,
\eeq
where $Ai(\cdot,2)$ is the second primitive of $Ai$,
and similarly for $\phi^{Ai,2}_+$ with $Ai$ replaced by $Ci$.
\end{lemma}

\Remark
We note that, by construction, $\phi^{Ai,2}_\pm(0)$ is of order $O(1)$.

\begin{proof}
Solutions to $Airy \, \circ \, \partial_y^2 = 0$ are directly obtained from solutions of $Airy = 0$ by integrating twice in space.
Let us start with $\phi^{Ai}_-$ and let us define $\phi^{Ai,1}_-$ by
$$
\phi^{Ai,1}_-(y) = \gamma \int_y^{+\infty} \phi^{Ai}_-(x) \, dx 
= \gamma \int_y^{+ \infty} Ai_a(x) \Bigl[ 1 +  O(\gamma^{-2}) \Bigr] \, dx.
$$
The leading term is given by 
$$
I :=  \gamma \int_y^{+ \infty} Ai_a(x) \, dx.
$$
Repeating the computations of the previous section, we define $z = g(x)$ and get
$$
I := \gamma \int_{z \ge g(y)} g'(x)^{-3/2} Ai \Bigl( \gamma(z - y_c) \Bigr) \, dz .
$$
We note that the $Ai$ function is rapidly decaying with $z$, which is not the case of $g'(x)$, thus, following the lines of Appendix \ref{appendix3},
$$
\phi^{Ai,1}_-(y) =  g'(y)^{-3/2} Ai \Bigl( \gamma (g(y) - y_c), 1 \Bigr) \Bigl[1 + O(\gamma^{-2}) \Bigr].
$$
We integrate once more and define
$$
\phi^{Ai,2}_-(y) = \gamma \int_y^{+\infty} \phi^{Ai,1}_-(z) \, dz ,
$$
Repeating the same computations leads to (\ref{exphiAi2}).
 We similarly construct $\phi^{Ai,1}_+$ and $\phi^{Ai,2}_+$ which have an exponential behavior at $+ \infty$.
\end{proof}

We now turn to the study of the Green function $G^{Ai \, \partial_y^2}$ of $Airy \circ \partial_y^2$.
This Green function  is the double primitive of $G^{Ai}$, starting from $+ \infty$.
Let us detail its properties. The first primitive $G^{Ai,1}$ of $G^{Ai}$ is defined by
$$
G^{Ai,1}(x,y) = \int_{z \ge y} G^{Ai}(x,z) \, dz.
$$
We recall that $G^{Ai}(x,y)$ is rapidly decreasing on both sides of $x$. 
As a consequence, $G^{Ai,1}$ is rapidly decreasing
when $y > x$. For $y < x$ however, it is no longer decreasing but remains of constant magnitude.
In both cases we have, using (\ref{boundintG}) and (\ref{boundintG3}), 
\beq \label{boundGAi1}
| G^{Ai,1}(x,y) | \lesssim \zeta(x) 1_{y < x} + \zeta(x)^{1/4} \zeta(y)^{1/4} e^{\theta(x) - \theta(y)} 1_{y > x}.
\eeq
We now have 
$$
G^{Ai \, \partial_y^2}(x,y) = \int_{z \ge y} G^{Ai,1}(x,z) \, dz.
$$
 As a consequence, $G^{Ai \, \partial_y^2}$ is rapidly decreasing for $y > x$ and behaves linearly for $y < x$.
 Note that
 $$
 I = \int_{z \ge y} e^{-\theta(z)} \, dz = \int_{z \ge y} e^{-2 \Re  \gamma^{3/2} (g(z) - y_c)^{3/2} / 3} \, dz,
 $$
 and making the change of variables $t = g(z)$,
 $$
| I | \le \int_{t \ge g(y)}  |g'(z)|^{-1} e^{- 2 \Re \gamma^{3/2} (t - y_c)^{3/2}/3} \, dt
 \lesssim {1 \over g'(y)} {e^{-\theta(y)}  \over \gamma^{3/2} (g(y) - y_c)^{1/2}}
 $$
 provided $\gamma (y - y_c)$ is large enough.
 This leads to
 $$
 I \lesssim \gamma^{-3/2} \zeta(x)^{1/2} e^{-\theta(y)}.
 $$
 Using (\ref{boundGAi1}), we have
\beq \label{boundGAi2}
| G^{Ai \, \partial_y^2}(x,y) | \lesssim  \zeta(x) (x -y ) 1_{y < x} 
+ \gamma^{-3/2}  \tilde \zeta(x) 1_{y < x} 
\eeq
$$
+ \gamma^{-3/2}  \zeta(x)^{1/4}  \tilde \zeta(y)^{3/4} e^{ \theta (x) - \theta(y) } 1_{y > x},
$$
where $\tilde \zeta (y) = \sup_{z > y} \tilde \zeta(y)$.

The solver ${\cal S}^{Ai  \, \partial_y^2}$ of $Airy \circ \partial_y^2$ is directly given by
$$
{\cal S}^{Ai  \, \partial_y^2} = \partial_y^{-2} \, {\cal S}^{Ai} 
$$
and is the convolution of the source with the Green function $G^{Ai \, \partial_y^2}$.
Let us now detail the smoothness of solutions to $Airy \circ \partial_y^2$.

 \begin{lemma} \label{solutionAiry2}
 Let $\sigma > 0$. Assume that $| y_c \log \gamma  | \lesssim 1$. 
 Let $f$ be an holomorphic function on $\Gamma_{\beta_0,\beta_1}$. 
 Then there exists an holomorphic function $\phi$, defined on $\Gamma_{\beta_0,\beta_1}$,
 such that
 \beq \label{eqq}
 Airy \, \partial_y^2 \phi = f,
 \eeq
 and such that
 \beq \label{bo1}
 | \phi(y) | \lesssim   e^{-\sigma y} \| f \|_{X_\sigma},
 \eeq
 \beq \label{bo2}
 | \partial_y \phi(y) | \lesssim | \log \gamma|  \,  e^{-\sigma y} \| f \|_{X_\sigma},
 \eeq
 \beq \label{bo3}
 | \partial_y^2 \phi(y) | \lesssim | \gamma |  e^{-\sigma y} \| f \|_{X_\sigma}
 \eeq
 \beq \label{bo3}
 | \partial_y^3 \phi(y) | \lesssim | \gamma |^2 e^{-\sigma y} \| f \|_{X_\sigma}
 \eeq
   and
\beq \label{bo5}
 | \partial_y^4 \phi(y) | \lesssim | \gamma |^3 e^{-\sigma y}  \| f \|_{X_\sigma}.
\eeq
 \end{lemma}

\begin{proof}
The estimates on $\partial_y^2 \phi$, $\partial_y^3 \phi$ and $\partial_y^4 \phi$ are direct consequences of Lemma \ref{solutionAiry}.
Now, let
$$
\zeta_1(y) = \int_y^{1/2} \tilde \zeta(z) \, dy.
$$
For $\Re y > \Re y_c + \gamma^{-1}$, we have $\zeta_1(y) \lesssim \log(y - y_c)$, and for $y < \Re y_c + \gamma_1$, we have
$\zeta_1(y) \lesssim \log \gamma$, which, using (\ref{bbb1}), gives the bound on $\partial_y \phi$.
Let now $\zeta_2$ be the second primitive of $\zeta_1$ which vanishes at $1/2$.
 Then $\zeta_2(y) \lesssim  1 +  |y - y_c| | \log(y - y_c)| $ 
for $\Re y > \Re y_c + \gamma$
and else $\zeta_2(y) \lesssim 1 + | \gamma || \log \gamma |  + | y_c| | \log \gamma |$, which ends the proof.
\end{proof}

 
 \subsection{A first inversion of $OS_{\alpha,c,\nu}$}
 
 
 We recall that we do not yet take into account the boundary conditions at $0$.
 Without its boundary conditions at $0$, $OS_{\alpha,c,\nu}$ is simply a fourth order differential equation and can thus always be
 "solved". The question of the spectrum only appears when we add the boundary conditions at $0$.
 
 For functions which are rapidly decaying in $y$, $OS_{\alpha,c,\nu}$ may be seen as a first order perturbation of $Airy \circ \partial_y^2$, which
 allows to solve $OS_{\alpha,c,\nu}$ in rapidly decreasing functions spaces. 
 
 \begin{proposition} \label{firstresolution}
 Assume that 
 \beq \label{hypoc}
 | y_c | \log |\gamma| \ll 1 .
 \eeq
 Then, provided $\sigma$ is large enough, for every $f \in X_\sigma$, there exists a solution $\phi \in X_\sigma$ to 
 $$
 OS_{\alpha,c,\nu} \phi = f
 $$
such that
\beq \label{boundOSS}
\| \phi \|_{X_\sigma} + {\| \partial_y \phi \|_{X_\sigma}  \over | \log \gamma |}  
+ { \| \partial_y^2  \phi \|_{X_\sigma}   \over | \gamma |}
+ {   \| \partial_y^3 \phi \|_{X_\sigma}  \over | \gamma |^2} 
+ { \| \partial_y^4 \phi \|_X  \over | \gamma^3|}   \lesssim \| f \|_{X_\sigma} .
\eeq
 \end{proposition}
 
 \Remark
 Note that the solution is not unique since we have not prescribed the boundary conditions at $0$.
 Condition (\ref{hypoc}) will be checked later.
 
 \begin{proof}
 
 We apply the second method described in Appendix \ref{appendix2}.
 In this case we have
 $$
 R(x,y) = \Bigl[ {\cal A}_0 \partial_y^{-2} {\cal S}^{Ai} \delta_x \Bigr](y) = {\cal A}_0(y) G^{Ai  \, \partial_y^2}(x,y).
 $$
 We recall that ${\cal A}_0(x)$ is bounded.
 Moreover, using (\ref{boundGAi2}), $G^{Ai \, \partial_y^2}$ is bounded by the sum of three terms. 
 We only focus on the first one, the other being smaller.
 We are thus lead to boud
 $$
 I = \int_y^{+\infty}  e^{-\sigma x} \zeta(x) (x-y)  \, dx.
$$
 The integral between $1$ and $+\infty$ is bounded by $C \sigma^{-1} e^{-\sigma y}$ since in this case $\zeta(x) = 1$.
  It remains to bound
 $$
 I_1 = \int_y^1 x e^{-\sigma x} \min(|x -y_c|^{-1}, \gamma)  \, dx .
 $$
We split $I_1$ in $I_2 + I_3$, where $I_2$ is the integral over $J = [y_c - \gamma^{-1},y_c+\gamma^{-1}] \cap [y,1]$, and
$I_3$ is the integral on its complementary $J^c = [y,1] - J$. 
First $I_2 = 0$ except if $y < y_c + \gamma^{-1}$. In this case
$$
| I_2 | \le 2 \gamma^{-1}   (y_c + \gamma^{-1})  e^{- \sigma y} \gamma
\le 2 (y_c + \gamma^{-1}) e^{- \sigma y}.
$$
Second, on $J^c$, $\zeta(x) = \max(|x - y_c|^{-1},1)$, thus we have
$$
| I_3 | \le \int_{J^c} e^{- \sigma x} \, dx + \int_{J^c} {| y_c | \over | x - y_c|} e^{- \sigma x} \, dx.  
$$
The first integral is bounded by $\sigma^{-1} e^{-\sigma y}$ and the second one  is bounded by 
$2 | y_c  | e^{-\sigma y} \log | \gamma^{-1}|$.

Thus assumption (\ref{hypo2}) of Proposition \ref{prophypo2} is thus satisfied provided $\sigma$ is large enough, which leads to the existence of $\phi$, with
$\| \phi \|_{X_\sigma} \lesssim \| f \|_{X_\sigma}$. Now
$$
Airy  \, \partial_y^2 \phi = f + {\cal A}_0 f.
$$
In particular, 
$$
\| Airy \, \partial_y^2 \phi \|_{X_\sigma} \lesssim \| f \|_{X_\sigma},
$$
and Lemma \ref{solutionAiry2} ends the proof.
\end{proof}


\subsection{Construction of "fast" solutions}


The aim of this section is to construct the "fast" modes $\phi_{f,\pm}$.  Let $Ai(x,1)$ be the first primitive of $Ai$ and $Ai(x,2)$ be the primitive
of $Ai(x,1)$.

\begin{proposition} \label{constructfast}
There exist two solutions $\phi_{f,\pm}$ of Orr Sommerfeld equations such that
$\phi_{f,-}$ converges exponentially fast to $0$ and $\phi_{f,+}$ diverges exponentially fast as $y \to + \infty$.
Moreover,
\beq \label{phimoins1}
\phi_{f,-}(0) = Ai( - \gamma y_c, 2) + O(\gamma^{-2}),
\eeq
\beq \label{phimoins2}
\partial_y \phi_{f,-}(0) =  \gamma Ai( - \gamma y_c, 1)+ O(\gamma^{-1}).
\eeq
In particular,
\beq \label{ddis}
{\partial_y \phi_{f,-}(0) \over \phi_{f,-}(0)} =  { \gamma Ai( - \gamma y_c, 1) \over Ai( - \gamma y_c, 2) } + O(\gamma^{-2}).
\eeq
\end{proposition}

\begin{proof}
  Let $\phi_0(y) = Ai(\gamma (y - y_c),2)$ and let $f = OS_{\alpha,c,\nu}(\phi_0) = - {\cal A}_0 \phi_0$, which is rapidly decreasing.
 Using Proposition \ref{firstresolution} we then solve $OS_{\alpha,c,\nu} \phi = f$ and set $\phi_{f,-} = \phi_0 + \phi$.
\end{proof}


\subsection{Remarks}


Let us make various remarks on Airy equations. First the adjoint of ${\cal A}$ is simply
\beq \label{adjointA}
{\cal A}^t = \partial_y^2 \circ Airy,
\eeq
thus two independent solutions of ${\cal A}^t$ can be obtained simply by differentiating twice $\phi^{Ai}_\pm$.

Next, for large $|c|$, $Airy$ equation is a perturbation of $- \partial_y^2 \psi - c \psi$ whose solutions are explicit. 
It is thus possible to write full expansions of $\phi^{Ai}_\pm$ as $c$ is large.

Note that up to now we have solved the various Airy equations with source terms in $\Gamma_{\beta_0,\beta_1}$. 
Using the Green function it is possible to solve the same equations with source terms in $\Gamma_{\beta_0,\beta_1}(y_c)$.
For instance the solution of $Airy(\phi) = \psi$ may be defined by
$$
\phi(y) = \int_\Gamma G^{Ai}(x,y) \psi(x) \, dx
$$
where $\Gamma$ is a path in $\Gamma_{\beta_0,\beta_1}(y_c)$ passing through $y$.


\section{Construction of approximate slow solutions}


We now construct approximate "slow" solutions of the Orr Sommerfeld equations, starting from the  solutions of Rayleigh equations constructed
in Theorem \ref{Rayexp}.

Let us first discuss the smoothness of $\phi_{s,-}$, the smoothness of $\phi_{s,+}$ being similar.
We first recall that the Rayleigh solutions $\psi_{\alpha,\pm}$ have a singularity at $y_c$, of the form $(y - y_c) \log(y - y_c) Q(y- y_c)$
where $Q$ is an holomorphic function near $y_c$.
However, $\phi_{s,-}$ is the solution to 
\beq \label{smooo}
- \eps \partial_y^4 \phi_{s,-} + Ray_{\alpha,c}(\phi_{s,-}) = 0
\eeq
where $Ray_{\alpha,c}$ is a second order differential operator. Thus $\phi_{s,-}$ is a solution of a regular fourth order differential equation, depending
in an holomorphic way on several parameters, namely $\alpha$ and $\nu$. 
As a consequence $\phi_{s,-}$ is an holomorphic function of $y$, $\alpha$ and $\nu$, and  has in particular {\it no} singularity at $y_c$.
Of course the bounds obtained by using (\ref{smooo}) are very poor and diverge as $\nu$ goes to $0$.

The smoothness of $\phi_{s,\pm}$ is a crucial point since it allows the migration of Dunford's 
 integration contours into the area $\Im c < 0$, which is not possible for Rayleigh equation where
we are limited by logarithmic branches. As a consequence, the decay for Euler solutions is only polynomial, in $\langle \alpha t \rangle^{-1}$,
whereas it is exponential, though with a slow rate, for Navier Stokes solutions.

\medskip

We split the construction of $\phi_{s,\pm}$ in two steps, since it seems delicate to remove all the singularities $(y - y_c)^n \log(y - y_c)$ appearing
in $\psi_{\alpha,\pm}$ at the same time. We thus first eliminate the largest ones, namely $(y-y_c)^n \log(y - y_c)$ for $n = 1$ to $4$,
which leads to the approximate solutions $\phi_{s,\pm}^{app}$. We later further eliminate the singularities through the iteration of the Green function.
Thus, we first prove:

\begin{proposition}\label{lem-exactphija} 
Let $N$ be arbitrarily large.
Assume that $ | y_c | \log |\gamma | \to 0$ as $\nu \to 0$. Then, for $\nu$ small enough,
there exist two independent approximate solutions $\phi_{s,\pm}^{app} \in \Gamma_{\beta_0,\beta_1}(y_c)$ 
to the Orr Sommerfeld equation with unit Wronskian,
such that $\phi_{s,-}^{app}(y) \to 0$  and $|\phi_{s,+}^{app}(y)| \to + \infty$ as $y \to + \infty$, and they satisfy
$$
Orr_{\alpha,c,\nu}(\phi_{s,\pm}^{app}) = O(\eps^N e^{\pm |\alpha| y}).
$$
Furthermore, we have the following expansion 
\beq \label{derivphis}
{\partial_y \phi_{s,-}^{app}(0) \over \phi_{s,-}^{app}(0)} = {\partial_y \psi_{-,\alpha}(0) \over \psi_{-,\alpha}(0)} + O(\eps)
\eeq
where ${\partial_y \psi_{-,\alpha}(0) / \psi_{-,\alpha}(0)}$ is given by (\ref{psippsi}).
Moreover,
$$
\| e^{- \alpha y} \phi_{s,+}^{app} \|_{L^\infty} + {\| e^{- \alpha y}  \partial_y \phi_{s,+}^{app} \|_{L^\infty}  \over | \log \gamma|} 
+ {\|  e^{- \alpha y} \partial_y^2 \phi_{s,+}^{app} \|_{L^\infty}  \over | \gamma |} 
+ {\|  e^{- \alpha y}  \partial_y^3 \phi_{s,+}^{app} \|_{L^\infty}  \over | \gamma|^2} \lesssim 1
$$
and
$$
\| e^{ \alpha y} \phi_{s,-}^{app} \|_{L^\infty} + {1 \over | c | \, | \log \gamma|} \| e^{ \alpha y}  \partial_y \phi_{s,-}^{app} \|_{L^\infty}   
$$
$$
+ { 1  \over | c | \,  | \gamma |} \|  e^{ \alpha y} \partial_y^2 \phi_{s,-}^{app} \|_{L^\infty}  
+ {1   \over | c | \,  | \gamma|^2}  \|  e^{ \alpha y}  \partial_y^3 \phi_{s,-}^{app} \|_{L^\infty} \lesssim 1 .
$$
 \end{proposition}

\Remarks
The estimates on $\phi_{s,-}^{app}$ are slightly better than that of $\phi_{s,+}^{app}$ since $\psi_{-,\alpha}$ is less singular than
$\psi_{+,\alpha}$ in $y_c$ as stated in (\ref{sizesingularity}).

We only get approximate solutions, since our construction method involves loss of derivatives at each iteration. We will recover genuine solutions
later, using the Green function.

Note that the loss of a $|\gamma|$ factor when we differentiate is coherent with the fact that we have chosen $\beta_0$ of order $|\gamma|^{-1}$.

\begin{proof}
Let us detail the construction of $\phi_{s,-}^{app}$, the construction of $\phi_{s,+}^{app}$ being similar.
We construct $\phi_{s,-}^{app}$ by iteration, starting from $\phi_0 = \psi_{\alpha,-}$. We recall that, near the critical layer, $\phi_0$ is of the form
$$
\phi_0 = P + (y - y_c) \log(y - y_c) Q
$$
where $P$ and $Q$ are holomorphic functions, and that $\phi_0$ is holomorphic away from the critical layer. We thus note that
$$
R_0 = Orr_{\alpha,c,\nu} \psi_0= - \eps \partial_y^4 \psi_0
$$
is holomorphic for large $y$, and decaying like $e^{- \alpha y}$, but of the form 
\beq \label{R01}
R_0 =  - {2 \eps Q \over (y - y_c)^3} - {4 \eps Q' \over (y - y_c)^2} -  {6 \eps Q'' \over y - y_c} - 4 \eps \log(y - y_c) Q'''
\eeq
$$
 - \eps (y - y_c) \log(y - y_c) Q^{(4)}  - \eps Q_1 - \eps P^{(4)}
$$
for  $y$ close to $y_c$, where $Q$ and $Q_1$ are holomorphic.
We note that $R_0$ is as singular as $\eps (y - y_0)^{-3}$.

The main step of the proof is to split $R_0$ between one part which will be solved using Airy's operator thanks to Proposition \ref{firstresolution}
and another  part which will be approximately solved using Rayleigh operator thanks to Proposition \ref{propXrho}.

We introduce
$\psi$ such that $\psi$ decays exponentially fast at infinity, and such that $1 - \psi$ has a zero of order $4$ at $y_c$. For instance we choose 
$$
\psi(y) = \exp( - \mu (y - y_c)^4),
$$
where $\mu$ will be chosen large enough.
We  then split $R_0$ into
$$
R_0 = R_{0,critical} + R_{0,Rayleigh}, 
$$
where
$$
R_{0,critical}(y) = \psi(y) R_0(y) 
$$
and
$$
R_{0,Rayleigh} = ( 1 - \psi(y) )  R_0.
$$
We note that $R_{0,Rayleigh}$ is of the form (\ref{ee1}) near $y_c$, holomorphic away from $y_c$, and decays like $e^{- \alpha y}$ at infinity.
We thus define $\phi_1$ to be the solution of
$$
Ray \, \phi_1 = - R_{0,Rayleigh} 
$$
which goes to $0$ at $+\infty$. Using Proposition \ref{propXrho}, $\phi_1$ exists and is of the form (\ref{ee2}) near $y_c$
and holomorphic away from $y_c$. It is exponentially decaying for large $y$, like $e^{-\alpha y}$, and is of order $O(\eps)$.

We now turn to $R_{0,critical}$ and want to solve
$$
Orr_{\alpha,c,\nu} \phi_s = - R_{0,critical} .
$$
We note that $R_{0,critical}$ is rapidly decreasing, however it is not bounded and has in fact the same singularity as $R_0$ at $y_c$.
More precisely, we expand $Q$ in power series near $y_c$ and rewrite $R_{0,critical}$ under the form 
\beq \label{criti}
R_{0,critical} =  - { \eps Q_0 \over (y - y_c)^3} - { \eps Q_1 \over (y - y_c)^2} -  {\eps Q_2 \over y - y_c} - Q_3 \eps \log(y - y_c) 
 - \eps \widetilde Q,
\eeq
where $Q_0$, $Q_1$, $Q_2$ and $Q_3$ are various constants, and where $\widetilde Q$ is a function of the form (\ref{ee1}).
Let $\eps Q_{sing}$ denote the four first terms of the right hand side of (\ref{criti}).
We then introduce $\phi_{sing}$, defined by
$$
\phi_{sing} = Q_0^1 (y - y_c) \log (y - y_c) + Q_1^1 (y - y_c)^2 \log (y - y_c)
$$
$$
+ Q_2^1 (y - y_c)^3 \log (y - y_c) + Q_3^1 (y - y_c)^4 \log (y - y_c)
$$
and choose the various coefficients $Q_0^1$, $Q_1^1$, $Q_2^1$ and $Q_3^1$ in such a way that the singular terms in 
$\eps \partial_y^4 \phi_{sing}$ exactly equal $- \eps Q_{sing}$, namely in such a way that 
$\eps \partial_y^4 \phi_{sing} + \eps Q_{sing}$ is of the form  (\ref{ee1}).

We then write $\phi_s =  \psi \phi_{sing} + \phi_{ref}$, where
\beq \label{previous}
Orr_{\alpha,c,\nu} (\phi_{ref}) =  \eps \widetilde Q - \eps Q_{sing} - Orr_{\alpha,c,\nu}(\psi \phi_{sing})
\eeq
As the right hand side of (\ref{previous}) is rapidly decreasing and bounded, we may apply Proposition \ref{firstresolution}.
As a consequence, $\phi_s$ is well defined on $\Gamma_{\beta_0,\beta_1}(y_c)$ and bounded on this domain.

We then introduce 
$$
\phi_{-,\alpha,1} = \phi_{-,\alpha} + \phi_1 + \phi_s + \phi_r
$$
and observe that 
$$
Orr_{\alpha,c,\nu}(\phi_{-,\alpha,1}) = - \eps \partial_y^4 \phi_1,
$$
which is of order $O(\eps^2)$.
We then iterate the process a finite number of times to end the proof of this Theorem. 
Note that this process can not be infinitely repeated since at each step we lose four derivatives.
\end{proof}


\section{Interior Green function}


The first step in the construction of the genuine Green function for Orr Sommerfeld equation is to construct an "interior" Green function,
namely a Green function which does not take into account boundary conditions at $y = 0$.
We first build an "approximate" Green function and then iterate it to get an "exact" one.


\subsection{Approximate interior Green function \label{approximateinterior}}


We first construct an approximate Green function $G_i^{app}$ which "approximately" solves
$$
Orr_{\alpha,c,\nu} G = \delta_x ,
$$
together with the classical boundary conditions at infinity. We do not enforce boundary conditions at $y = 0$.

We recall that $\mu(y)$, defined in (\ref{definitionmu}), satisfies
$$
\mu(y) \sim \gamma^{3/2} g'(y) (g(y) - y_c)^{1/2}, 
$$
provided $\gamma (y - y_c)$ is large enough, and is of order $\gamma$ if $\gamma (y - y_c)$ is bounded.
Note that $\mu(y)$ is of magnitude $\gamma^{3/2}$ when $y$ goes to infinity.

We look for $G_i^{app}$ under the form
\beq \label{int}
G_i^{app}(x,y) = a_+(x)  \phi_{s,+}^{app}(y) 
+ {b_+(x) }  {\phi_{f,+}(y) \over \phi_{f,+}(x)} 
\quad \hbox{for} \quad y < x,
\eeq
$$
G_i^{app}(x,y) = a_-(x)  \phi_{s,-}^{app}(y)
+ {b_-(x)} {\phi_{f,-}(y) \over \phi_{f,-}(x)} 
\quad \hbox{for} \quad y  > x.
$$

\begin{theorem} \label{boundG}
Let $N$ be arbitrarily large. 
Then we can choose $a_\pm$ and $b_\pm$ such that
\beq \label{appGreen0}
a_\pm(x) = O \Bigl( {e^{\mp | \alpha | x} \over \eps \mu^2(x)} \Bigr), \qquad 
b_\pm (x)= O \Bigl( {1 \over \eps \mu^3(x)} \Bigr),
\eeq
and such that
\beq \label{appGreen}
\Bigl| Orr_{\alpha,c,\nu} \, G_i^{app} - \delta_x \Bigr| \lesssim \nu^N e^{- |\alpha| |y - x|} .
\eeq
\end{theorem}

\Remarks
The approximate Green function $G_i^{app}$ uses the functions $\phi_{s,\pm}^{app}$ which are only approximate solutions, which leads
to the right hand side term in (\ref{appGreen}).

Let us first give an heuristic proof of this Theorem when $\gamma (y - y_c)$ is large. As a first approximation,
the jump in the third derivative of $G_i$ at $x$ (namely $- \eps^{-1}$), together with the continuity of its second derivative
are taken in charge by the fast modes $\phi_{f,\pm}$. As the third derivatives of these modes are of order $\gamma^{9/2}$, this
implies that $b_\pm$ are of order $\eps^{-1} \gamma^{-9/2} = O( \gamma^{-3/2})$. The continuity of $G_i$ and its first derivatives are taken in charge by
the slow modes $\phi_{s,\pm}^{app}$. First derivatives of fast modes are of order $\eps^{-1} \gamma^{-3}$, hence $a_\pm$ are of order
$\eps^{-1} \gamma^{-3} = O(1)$, namely of order $\eps^{-1} \mu^{-2}(x)$.

The scaling is different when $\gamma (y - y_c)$ is bounded. In this case, the third derivatives of the fast modes are of order $\gamma^3$ only, thus
$b_\pm$ are of order $\eps ^{-1} \gamma^{-3} = O(1)$. The first derivatives are of order $\eps^{-1} \gamma^{-2}$, thus $a_\pm$ are
of order $\eps^{-1} \gamma^{-2} = O(\gamma)$, namely much larger, which is natural since at $y_c$ the second order term $(U_s - c) \partial_y^2$
almost vanishes.

\begin{proof}
The proof follows the lines of \cite{GN}. Let 
$$
v(x) = \Bigl( - a_-(x), a_+(x), - b_-(x), b_+(x)  \Bigr) ,
$$
and let 
$$
M = \left( \begin{array}{cccc} 
\phi_{s,-}^{app} & \phi_{s,+}^{app} & 1  &  1  \cr
\partial_y \phi_{s,-}^{app}  / \mu & \partial_y \phi_{s,+}^{app} /   \mu
& \partial_y\phi_{f,-} /  \phi_{f,-} \mu & \partial_y\phi_{f,+} /   \phi_{f,+} \mu  \cr
\partial_y^2 \phi_{s,-}^{app} /   \mu^2 &\partial_y^2 \phi_{s,+}^{app} /   \mu^2
& \partial_y^2 \phi_{f,-} /  \phi_{f,-}  \mu^2 & \partial_y^2 \phi_{f,+} /   \phi_{f,+} \mu^2 \cr
 \partial_y^3 \phi_{s,-}^{app} /   \mu^3 &  \partial_y^3 \phi_{s,+}^{app} /  \mu^3 
& \partial_y^3 \phi_{f,-}  /  \phi_{f,-} \mu^3 &  \partial_y^3 \phi_{f,+} /   \phi_{f,+} \mu^3  \cr 
\end{array} \right) ,
$$
where the functions $\phi_{s,\pm}^{app}$ and $\phi_{f,\pm}$ are evaluated at $x$.
By definition of the Green function, $G_i^{app}$, $\partial_y G_i^{app}$ and $\partial_y^2 G^{app}_i$ are continuous at $x = y$,
whereas $- \eps \partial_y^3 G_i^{app}$ has a unit jump at $x = y$,
thus 
\beq \label{Mv}
M v = (0,0,0,- 1/ \eps \mu^3(x)) .
\eeq
To solve this equation we use the $2 \times 2$ block structure of $M$.
Let $A$, $B$, $C$ and $D$ be the two by two matrices defined by
$$
M = \left( \begin{array}{cc} 
A & B \cr
C & D \cr \end{array} \right) .
$$
We will prove that $D$ is invertible, $C$ is small, $B$ is bounded and $A$ is invertible, and then invert $M$ by iteration.

Let us first study $B$ and $D$, namely the third and fourth columns of $M$, which are related to "fast solutions" constructed using Airy's functions. 
First, when $\gamma (x - y_c)$ is large, we have
$$
{\partial_y \phi_{f,-} \over \phi_{f,-}} =  g'(x) \gamma {Ai(\gamma (g(x) - y_c),1) \over Ai( \gamma (g(x) - y_c),2)} 
\Bigl[ 1 + O(\gamma^{-1}) \Bigr]
$$
$$
\sim  \gamma^{3/2} g'(x) (g(x) - x_c)^{1/2}  \Bigl[ 1 + O(\gamma^{-1}) \Bigr] \sim \mu(x),
$$
$$
{\partial_y^2 \phi_{f,-} \over \phi_{f,-}} =  g'(x)^2 \gamma^2 {Ai(\gamma (g(x) - y_c)) \over Ai( \gamma (g(x) - y_c),2)} 
\Bigl[ 1 + O(\gamma^{-1}) \Bigr]
$$
$$
\sim  \gamma^3 g'(x)^2 (g(x) - y_c)  \Bigl[ 1 + O(\gamma^{-1}) \Bigr] \sim \mu^2(x),
$$
$$
{\partial_y^3 \phi_{f,-} \over \phi_{f,-}} =  g'(x)^3 \gamma^3 {Ai'(\gamma (g(x) - y_c)) \over Ai( \gamma (g(x) - y_c),2)}
\Bigl[ 1 + O(\gamma^{-1}) \Bigr]
$$
$$
\sim  \gamma^{5/2} g'(x)^3 (g(x) - y_c)^{3/2}  \Bigl[ 1 + O(\gamma^{-1}) \Bigr] \sim \mu^3(x),
$$
and similarly for $Ci$, up to a minus sign for odd derivatives. Thus
we get
$$
D = \left( \begin{array}{cc}
1 & 1 \cr
-1 & 1 \cr 
\end{array} \right) + o(\gamma^{-1}),
$$
hence $D$ is invertible.
For small or bounded $\gamma (y - y_c)$, we note that $\phi_{f,\pm}(x)$ is of order $O(1)$ and that
$$
D =  \Bigl[ 1 + O(\gamma^{-1})  \Bigr] 
\left( \begin{array}{cc}
{g'(y)^2 \gamma^2 \mu^{-2}(x) \over Ai( \gamma (g(y) - y_c),2)}    & 0 \cr
0 & {g'(y)^3  \gamma^3 \mu^{-3}(x) \over Ci( \gamma (g(y) - y_c),2)} \cr 
\end{array} \right) 
\left( \begin{array}{cc}
Ai & Ci \cr
Ai' & Ci' \cr 
\end{array} \right)  .
$$
As $g'(y_c) = 1$ the first array on the right hand side is close to the identity, and the second one is invertible and bounded by definition of
$Ai$ and $Ci$.
Using similar arguments, we see that $B$ is bounded.

Let us turn to $A$ and $C$ which are related to the "inviscid" and "slow" functions $\phi_{s,\pm}^{app}$.
First, using Proposition \ref{lem-exactphija},
$$
C =  \left( \begin{array}{cc}  
O(c \gamma  / \mu^2)  &O( \gamma / \mu^2) \cr
O(c \gamma^2  / \mu^3) & O( \gamma^2 / \mu^3) \cr 
\end{array} \right).
$$
For small $\gamma (x - y_c)$, $\gamma / \mu^2$ and $\gamma / \mu^3$ are of order $\gamma^{-1}$.
For large $\gamma (x - y_c)$, $\gamma / \mu^2$ is of order $\gamma^{-1} g'(x)^{-2} (g(x) - y_c)^{-1}$ which is of order $\gamma^{-2}$,
and similarly for $\gamma^2 / \mu^3$. Thus
\beq \label{sizeC}
C = O(\gamma^{-1}).
\eeq
Let us turn to $A$.
Note that $A = A_1 A_2$ with
$$
A_1 = \left( \begin{array}{cc}  
1 & 0 \cr
0 & \mu^{-1} \cr
\end{array} \right),
\quad 
A_2 = \left( \begin{array}{cc}  
\phi_{s,-}^{app} & \phi_{s,+} ^{app}\cr
\partial_y \phi_{s,-}^{app}  & \partial_y \phi_{s,+}^{app}  \cr
\end{array} \right) .
$$
We have
$$
A^{-1}  = {1 \over \det(A_2)} \left( \begin{array}{cc}  
\partial_y \phi_{s,+}^{app}  & - \phi_{s,+}^{app} \cr
- \partial_y \phi_{s,-}^{app} & \phi_{s,-}^{app}  \cr
\end{array} \right) 
 \left( \begin{array}{cc}  
1 & 0 \cr
0 & \mu \cr
\end{array} \right).
$$
Hence, for small $x$,
$$
A^{-1}
= {1 \over \det(A_2)} \left( \begin{array}{cc}  
 O(\log | \gamma |)  & O(1)  \cr
O(c \log |\gamma |) & o(1)  \cr
\end{array} \right) 
 \left( \begin{array}{cc}  
1 & 0 \cr
0 & \mu \cr
\end{array} \right)
$$
\beq \label{sizeAinv}
= {1 \over \det(A_2)} \left( \begin{array}{cc}  
 O(\log | \gamma|)  & O(\mu)  \cr
O(c \log | \gamma |) & o(\mu) \cr
\end{array} \right) .
\eeq
The determinant of $A_2$ is the Wronskian of $\phi_{s,\pm}^{app}$ which approximately equals $1$.
Now for $y$ away from $0$, 
$$
A^{-1} =  {1 \over \det(A_2)} \left( \begin{array}{cc}  
 O(\alpha e^{|\alpha| x} )  & O(\mu e^{ |\alpha| x})  \cr
O(\alpha e^{-|\alpha| x}) &  O(\mu e^{- |\alpha| x})  \cr
\end{array} \right) .
$$
We now turn to the inversion of the matrix $M$. We observe that the matrix $M$ has an approximate inverse
$$
\widetilde M = \left( \begin{array}{cc}
A^{-1} & - A^{-1} B D^{-1}  \cr
0 & D^{-1} \cr 
\end{array} \right) 
$$
in the sense that $M  \widetilde M = Id + N$ where
$$
N =  \left( \begin{array}{cc}
0 & 0 \cr 
 C A^{-1} &  - C A^{-1} B D^{-1} \cr 
\end{array} \right) .
$$
Now a direct calculation using (\ref{sizeC}) and (\ref{sizeAinv}) shows that, for small $x$, 
$$
C A^{-1} = o(1),
$$ 
and that, for large $x$, $CA^{-1} = o( e^{| \alpha | x})$.
Similarly
$$
C A^{-1} B D^{-1} = o( e^{| \alpha | x}).
$$
Thus $N = o(1)$. In particular, $(Id + N)^{-1}$
is well defined and 
$$
M^{-1} = \widetilde M (Id + N)^{-1} = \widetilde M \sum_{n\ge 0} N^n .
$$
Note that the first two lines of  $\sum_{n \ge 0} N^n$  vanish.
Therefore
$$
(Id + N)^{-1} (0,0,0, 1 / \eps \mu^3) = \Bigl( 0, 0, O(1 / \eps \mu^3),  O(1 / \eps \mu^3)  \Bigr) .
$$
As $D^{-1}$ is bounded we obtain that $b_\pm$ is of order $\eps^{-1} \mu^{-3}$.
Moreover, $A^{-1} B D^{-1}$ is of order $O(\mu)$, which gives the desired bound on $a_\pm$.
\end{proof}


\subsection{Exact interior Green function}


The construction of an exact Green function, starting from an accurate approximate one, using an iterative process, is classical,
and we will only sketch the proof of the following Proposition.

\begin{proposition}
Let $G^{int}$ be the exact interior Green function  for the Orr Sommerfeld equation, defined by
$$
Orr_{\alpha,c,\nu} \, G^{int} = \delta_x
$$
together with its boundary condition at infinity.
Then
$$
\| G^{int} - G_i^{app} \|_{X_{\alpha/2}} \lesssim \nu^N .
$$
\end{proposition}

\begin{proof}
Let 
$$
R(x,y) = Orr_{\alpha,c,\nu} G_i^{app}(x,y) - \delta_x 
$$
be the error of the approximate Green function. By construction, we have
$$
\| R(x,\cdot) \|_{X_{\alpha/2}} \le \nu^N 
$$
for some large integer $N$.
We construct the Green function $G^{int}$  by iteration, starting from $G_0 = G_i^{app}$.
 We have
 $$
 Orr_{\alpha,c,\nu} G_0 - \delta_x =  R_0(x,y)
 $$
 with $R_0 =R$.
 We then introduce 
 $$
 G_1(x,y) = - \int G_i^{app}(z,y) R(x,z) \, dz .
 $$
 We note that $\| G_1(x,y) \|_{X_{\alpha/2}} \lesssim \nu^{N/2}$ provided $N$ is large enough.
 We then observe that
 $$
 Orr_{\alpha,c,\nu} G_1 + R_0(x,y) = R_1(x,y)
$$
where
$$
R_1(x,y) = - \int R(z,y) R(x,z) \, dz
$$
satisfies $\| R_1 \|_{X_{\alpha/2}} \lesssim \nu^{3N/2}$.
We then iterate the procedure and define
$$
G^{int}= G_i^{app} + \sum_{k \ge 1} G_k,
$$ 
which ends the proof.
\end{proof}

 
 \subsection{Construction of slow solutions}
 
 
 We now construct genuine slow solutions $\phi_{s,\pm}$, starting from approximate one, using the exact Green function constructed before.
 The method is classical, and we will only sketch the proof of the following result.
 
 \begin{proposition}
 There exists two solutions $\phi_{s,\pm} \in \Gamma_{\beta_0,\beta_1}$ to Orr Sommerfeld equations, which behave like $e^{\pm |\alpha| y}$ at infinity.
 These solutions satisfy the estimates of Proposition \ref{lem-exactphija}.
 \end{proposition}

 \Remarks
 Note that we are not dealing with modes of Orr Sommerfeld equations, but simply of solutions, which allows the construction through the Green function.
 
 As all the solutions of Orr Sommerfeld are smooth, $\phi_{s,\pm}$ is smooth, even at $y_c$, and thus defined on $\Gamma_{\beta_0,\beta_1}$.
 
 \begin{proof}
 The construction of genuine slow solutions is done by iteration, following the lines of Proposition \ref{constructfast}.
 As observed in the beginning of section $5$, the solution is smooth and thus belongs to $\Gamma_{\beta_0,\beta_1}$, and
 not only to $\Gamma_{\beta_0,\beta_1}(y_c)$.
  \end{proof}

 
 \subsection{Second construction of $G^{int}$}
 
 
 Now that we have constructed four exact solutions of Orr Sommerfeld equations, namely $\phi_{f,\pm}$ and $\phi_{s,\pm}$, we can
 restart the construction of section \ref{approximateinterior},  with $\phi_{s,\pm}^{app}$ replaced by $\phi_{s,\pm}$. The construction
 is similar and Theorem \ref{boundG} holds true with approximate solutions replaced by true ones in the definition of the Green function.


\section{Discussion of the dispersion relation}



\subsection{Spectrum of Orr Sommerfeld operator}


We now prove the existence of unstable eigenvalues in some range of $\alpha$. We only focus on small $| \alpha |$.


\subsubsection{The dispersion relation}


Let $c$ be an eigenvalue of Orr Sommerfeld equation, with associated eigenmode $\psi$.
As $\psi$ goes to $0$ at $+ \infty$, it must be a linear combination of $\psi_{s,-}$ and $\psi_{f,-}$, namely
$$
\psi = a \phi_{s,-} + b \phi_{f,-} 
$$
for some non zero constants $a$ and $b$.
Moreover, we must have $\psi(0) = \psi'(0) = 0$. This leads to the dispersion relation
\beq \label{disper}
{\partial_y \phi_{s,-}(0) \over \phi_{s,-}(0)} = {\partial_y \phi_{f,-}(0) \over \phi_{f,-}(0)} ,
\eeq
or, equivalently
\beq \label{disperr}
{\cal E}(\alpha,\nu,c) := W[\phi_{s,-},\phi_{f,-}](0) = 0 .
\eeq
Note that the left hand side of (\ref{disper}) is, 
at leading order, given by (\ref{psippsi}), namely by Rayleigh {\it inviscid} part,
whereas the right hand side, given by (\ref{ddis}), only depends on the Airy {\it viscous} part (see \cite{Reid}, page $273$, equation $3.73$).
This dispersion relation has of course been intensively studied in physics. We follow the presentation given by W.H. Reid in \cite{Reid} and 
his notations.
Let
$$
w = \Bigl[ 1 + {U_s'(0) \phi_{s,-}(0) \over c \phi_{s,-}'(0)} \Bigr]^{-1}
$$
and let
$$
\Lambda  = {U_s'(0) \over c} y_c - 1 .
$$
Note that $\Lambda = O(c)$.
Let
\beq \label{defiz}
z = \Bigl( {\alpha U_c' \over \nu} \Bigr)^{1/3} y_c, \qquad \xi_1  = - i^{1/3} z, 
\eeq
and let the Tietjens function be defined by 
\beq \label{Tietjensdefi}
Ti(z) = {Ai(\xi_1,2) \over \xi_1 Ai(\xi_1,1) } 
\eeq
as in Appendix \ref{appendix1}.
Then the dispersion relation (\ref{disper}) may be rewritten
\beq \label{disper2}
{w - 1 \over (1 + \Lambda) w} = Ti(z) \Bigl[ 1 + O(\gamma^{-2}) \Bigr]
\eeq
or equivalently
$$
(1 + \Lambda) Ti(z) = {U_s'(0) \over c} {\phi_{s,-}(0) \over \phi_{s,-}'(0)}  \Bigl[ 1 + O(\gamma^{-2}) \Bigr]. 
$$
Combining (\ref{derivphis}) and (\ref{psippsi}) we get
\beq \label{disper3}
(1 + \Lambda) Ti(z)  \Bigl[ 1 + O(\gamma^{-2}) \Bigr]
\eeq
$$
 = 1 - {\alpha \over c} {(U_+ - c)^2 \over  U_s'(0)} + {\alpha^2 \over c} {(U_+ - c)^4 \over  U_s'(0)}
\Bigl[ {1 \over U_s'(0) c} + \Omega_0(0,c) \Bigr] + O \Bigl( { \alpha^3 \over c} \Bigr) .
$$
We now focus on the study of this dispersion relation.


\subsubsection{Marginally stable modes and numerical illustration}


Let us first search for marginally stable modes, namely modes such that $\Im c = 0$. 
The following lines are classical in physics \cite{Drazin,Drazin2,Reid}.
Then $y_c$ is real, and thus $z$ is real. Using (\ref{defiOmega0}) we have
$$
\Re \Omega_0(0,c) = - {1 \over U_s'(0) c} + O (\log c).
$$
As $\Lambda = O(c)$, this leads to
\beq \label{reTi}
\Re Ti(z) = 1 - {U_+^2 \over U_s'(0)} {\alpha \over c}  + O \Bigl( { \alpha^2 \log c \over c} \Bigr) + O(\gamma^{-2})
\eeq
and, using (\ref{imomega0}),
\beq \label{imTi}
\Im Ti(z) \sim - \pi {\alpha^2 \over  c} {U_s''(0) U_+^4 \over U_s'^4(0)} .
\eeq
As the Tietjens function is bounded on the real line, this implies that $\alpha c^{-1}$ is bounded. Hence, as $\alpha \to 0$, $\Im Ti(z) \to 0$.
Thus, either $z \to + \infty$ or $z \to z_0 \sim 2.297$ (see Appendix \ref{appendix1}).

Let us study the first possibility. In this case, as $z \to + \infty$, $\Re Ti(z) \to 0$ and thus
$$
c \sim { U_+^2 \over  U_s'(0)} \alpha.
$$
Moreover, as $z \to + \infty$, by  (\ref{Tietjensinfini}),  $Ti(z) \sim - e^{i \pi / 4}  z^{-3/2}$. 
Using (\ref{defiz}), (\ref{imTi}) and $y_c \sim U_s'(0)^{-1} c$, this leads to $\alpha \sim C_+^\alpha \nu^{1/6}$, 
$c \sim C_+^c \nu^{1/6}$ and $z_c \sim C_+^z \nu^{1/6}$ for some constant $C_+^\alpha$, $C^c_+$ and $C_+^z$ which can
ben explicitly computed.
Moreover, $\gamma \sim C_+^\gamma \nu^{-5/18}$.
The critical layer is at a distance of order $\nu^{1/6}$ from the boundary. Its size, of order $\gamma^{-1} \sim \nu^{5/18}$, is much smaller
than its distance from the boundary: the critical layer is "split" from the boundary.
This case is referred to as the "upper marginal stability" regime (see Figure \ref{figure1}). 

Let us now turn to the second possibility, namely to the case when $z \to z_0$. In this case,
$$
c \sim C_0 {U_+^2 \over U_s'(0)} \alpha
$$
for some constant $C_0$.
Using $z \to z_0$ together with (\ref{defiz}), we obtain $\alpha \sim C_-^\alpha \nu^{1/4}$ and $c \sim C_-^c \nu^{1/4}$.
In this case,  $y_c \sim C_-^y \nu^{1/4}$ and $\gamma \sim C_-^\gamma \nu^{-1/4}$.
The critical layer is at a distance $\nu^{1/4}$ from the boundary, and its size is of the same magnitude. 
 This case is referred to as the "lower marginal stability" regime (see Figure \ref{figure1}).

Let us turn to numerical illustrations. In the case $U_+ = 1$, 
on the lower marginal stability branch we have (see \cite{Reid}, page $282$),
\beq \label{branch1}
c \approx 2.296 {\alpha \over U_s'(0)}, \qquad
\nu^{-1} \approx 1.002 \, U_s'(0) \alpha^{-4} .
\eeq
On the upper marginal stability branch (see \cite{Reid}, page $281$), when $U_+ = 1$, we have
\beq \label{branch2}
\nu^{-1} \approx {1 \over 2 \pi^2} {U_s'(0)^{11} \over U_1''(0)^2} \alpha^{-6}.
\eeq
When $\alpha$ is given, (\ref{disper3}) is easily solved using for instance Newton's method.
Figure \ref{figure2} displays $\Re \lambda$ as a function of $\alpha_0 = \alpha / \nu^{1/4}$ in the case
$U_+ = U_s'(0) = 1$. The maximum of $\Re \lambda$ occurs for $\alpha \approx 2.7 .  \nu^{1/4}$.

\begin{figure} \label{figure2}
\centerline{\includegraphics[width=7cm]{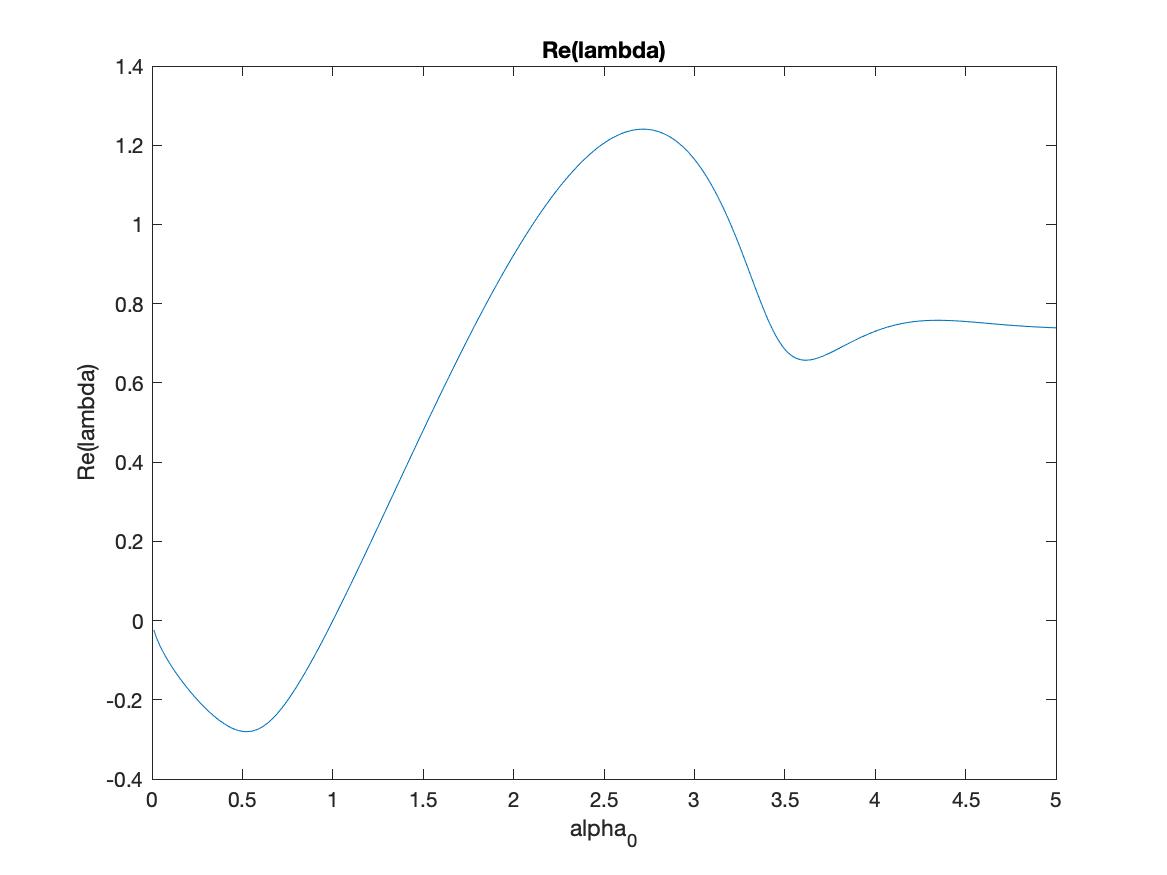}}
\caption{$\Re \lambda$ as a function of $\alpha_0 = \alpha / \nu^{1/4}$ in the case $U_+ = U_s'(0) = 1$.}
\label{lambdac}
\end{figure}


\subsubsection{Study of the dispersion relation}


Let us now study the dispersion relation (\ref{disper3}).  Let $\theta$ be defined by
$\alpha = \theta c$. Then (\ref{disper3}) reads
\beq \label{disper4}
(1 + \Lambda) \Bigl[ 1 + O(\gamma^{-2}) \Bigr] Ti(z) 
= 1 - A \theta + B \theta \alpha + O(\theta \alpha^2)
\eeq
where $A$ and $B$ are the two terms appearing in (\ref{disper3}). Note that $B$ depends on $c$ and is of order at most
$$
O(1) + O \Bigl( {\alpha^2 \log c \over c} \Bigr) = O(1) + O(\theta^2 c \log c) = O(1 + \theta^2).
$$
As $| \alpha | \ll 1$ and as $Ti$ is a bounded function, we note that $\theta$ is bounded.

If $z$ is bounded, as $\alpha$ is of order $c$, this implies that both $\alpha$ and $c$ are of order $O(\nu^{1/4})$.
The growth rate $\lambda = - i \alpha c$ then has a real part of order $\Re \lambda$ of order $O(\nu^{1/2})$.

We thus focus on the case where $z$ is large. Then $Ti(z) \sim - e^{i \pi / 4} z^{-3/2}$, thus (\ref{disper4}) gives
\beq \label{disper5}
- e^{i \pi / 4} {  \nu^{1/2}  \theta^{3/2}  U_c' \over \alpha^2 } \approx 1 - A \theta + B \theta \alpha + O(\theta \alpha^2).
\eeq
At first approximation, $\theta \sim A^{-1}$, which is a real number. As a second approximation,
\beq \label{disper6}
\theta \approx A^{-1} + B A^{-1} \alpha +  e^{i \pi / 4} {  \nu^{1/2}  U_c' \over  A^{3/2} \alpha^2 }.
\eeq
The imaginary part of $\theta$ changes sign when the imaginary part of the sum of the second and third terms cancel, namely 
when $\alpha$ is of order $\nu^{1/6}$ as previously described.
In this case, $\Re \lambda = \alpha \Im c = \alpha^2 \Im \theta^{-1}$ which is of order $\alpha^3$ is the unstable region, according to
(\ref{disper6}), namely at most of order $\nu^{1/2}$. 
Note that there exists no other solution $\theta$ of (\ref{disper5}) when $z$ is large.


\subsubsection{Evaluation of the Wronskian}


This section is devoted to the study of  $W^{-1}[\phi_{s,-},\phi_{f,-}]$.

\begin{proposition} \label{uniqueness}
All the eigenvalues $c(\alpha,\nu)$ are simple, unique in the area $\Im c \ge C_0 |\gamma|^{-1}$ provided
$C_0$ is small enough, and near them we have
\beq \label{boundWronskian}
\Bigl| {1 \over W[\phi_{s,-},\phi_{f,-}]} \Bigr| \lesssim { 1 \over \gamma | c - c(\alpha,\nu) |} .
\eeq
\end{proposition}

\Remark
Part of this Proposition relies on numerical computations to check that the Wronskian does not vanish when $\alpha$ is of order $\nu^{1/4}$.

\begin{proof}
We note that
$$
W[\phi_{s,-},\phi_{f,-}] =
\phi_{s,-}(0) \partial_y \phi_{f,-}(0) - \partial_y \phi_{s,-}(0)  \phi_{f,-}(0)  .
$$
We can choose  $\phi_{s,-}$ such that $\partial_y \phi_{s,-}(0) = 1$. In that case,
$$
\phi_{s,-}(0)=  -  c \Psi(\alpha,c)
$$
where
$$
\Psi(\alpha,c) = \Bigl[ U_s'(0) + \alpha c^{-1} (U_+ - c)^2 - \alpha^2 c^{-1} \Omega_0(0,c) + O(\alpha^3 c^{-1}) \Bigr]^{-1}.
$$
We recall that
$$
\phi_{f,-}(0) = Ai\Bigl( - \gamma y_c,2 \Bigr) \Bigl[ 1 + O(\gamma^{-2}) \Bigr]
$$
and
$$
\partial_y \phi_{f,-}(0) = \gamma Ai\Bigl( - \gamma y_c,1 \Bigr)  \Bigl[ 1 + O(\gamma^{-2}) \Bigr],
$$
 thus 
 $$
 \partial_c \phi_{f,-}(0) =  -   \gamma U_s'(0)^{-1}  Ai( - \gamma y_c,1) \Bigl[1  + O(\gamma^{-2}) \Bigr]
 $$
  and
 $$
 \partial_c \partial_y \phi_{f,-}(0) =  \gamma^2 U_s'(0)^{-1}  Ai(- \gamma y_c) \Bigl[ 1  + O(\gamma^{-2}) \Bigr] .
 $$
This leads to
$$
\partial_c W = \partial_c \phi_{s,-}(0) \partial_y  \phi_{f,-}(0) + \phi_{s,-}(0) \partial_c \partial_y \phi_{f,-}(0) 
- \partial_c \phi_{f,-}(0)
$$
$$
= \gamma  \Bigl( U_s'(0)^{-1} -  \Psi - c \partial_c \Psi \Bigr) Ai(-\gamma y_c,1)  - c \gamma^2 \Psi U_s'(0)^{-1} Ai(- \gamma y_c)  .
$$
Two cases arise. First is $\gamma$ is of order $\nu^{-1/4}$, namely if $z$ is bounded, then $\partial_c W$ is of order $\nu^{-1/4}$,
namely of order $O(\gamma)$.
It can be  {\it numerically} shown that $\partial_c W$ never vanishes. 
Similar computations show that $\partial_c^2 W$ is of order $\nu^{-1/2}$ which ends the proof in this particular case.

If $c \gamma \gg 1$ then $\partial_c W$ does not vanish, and is of order $O(c \gamma^2)$, 
which implies (\ref{boundWronskian}) and the fact that the eigenvalue is simple.
\end{proof}


\subsubsection{Description of the eigenvalues}


We now completely describe the eigenvalues in the area $\nu^{1/4}$ to $\nu^{1/6}$.
We have to solve
\beq \label{follow}
{\cal E}(\alpha,\nu,c) = 0
\eeq
and look for $c$ as a function of $\alpha$, $\nu$ being fixed.
In a neighborhood of the area of interest, $\partial_c W$ does not vanish.
As a consequence, we can apply the implicit function theorem to (\ref{follow}), and express $c$ as a function of $\alpha$.
Note that $\Im c(\alpha)$  must be negative for large enough $\alpha$.

 If $\Im c(\alpha)$ changes sign, then it must change sign at
$\alpha = O(\nu^{1/6})$, which gives the uniqueness of the branch containing unstable modes.
Now, using Proposition \ref{uniqueness}, there can be only one branch with $\Im c \ge C_0 \gamma^{-1}$, provided $C_0$ is small enough.


\subsection{Description of unstable modes}


Let us now detail the linear unstable mode $\psi_{lin}$ near the lower marginal stability curve $\alpha = O(\nu^{1/4})$ to fix the ideas. 
Its stream function $\psi_{lin}$ is of the form
$$
\psi_{lin} = \psi_{s,-} + a \psi_{f,-} .
$$
As $\psi_{lin}(0) = 0$, we have $a = O(\gamma^{-1}) = O(\nu^{1/4})$.
Hence, for bounded $y$,
\beq \label{psilin}
\psi_{lin}(y) = U_s(y) - c + \alpha {U_+^2 \over U_s'(0)} + a Ai \Bigl( \gamma (y - y_c),2 \Bigr) +  O(\nu^{1/2} ).
\eeq
The corresponding horizontal and vertical velocities $u_{lin}$ and $v_{lin}$ are given by
\beq \label{vlin}
u_{lin}(y) = \partial_y \psi_{lin} = U_s'(y) +\gamma a Ai \Bigl( \gamma (y - y_c),1 \Bigr) + O(\nu^{1/4}),
\eeq
and 
\beq \label{ulin} 
v_{lin}(y) = -i\alpha \psi_{lin}(y) = O(\nu^{1/4}).
\eeq
Note that $\gamma a = O(1)$. The first term in (\ref{vlin}) may be seen as a "displacement velocity", corresponding
to a translation of $U_s$. The second term is of order $O(1)$ and located near the critical layer.
 Note that the vorticity
\beq \label{check-use-number}
\omega_{lin}(y) = -(\partial_y^2 - \alpha^2) \psi_{lin}(y) 
= -U_s''(y) - \gamma^2 a Ai \Bigl( \gamma (y - y_c) \Bigr) + O(1)
\eeq
is large in the critical layer (of order $O(\gamma)$).

For large $y$, $\psi_{lin}$ decays like $e^{- \alpha y}$ and $u_{lin}$ and $v_{lin}$ like $\alpha e^{- \alpha y}$.
The decay of $\omega_{lin}$ is faster, like $e^{- \beta y}$.

\Remarks
This construction is sometimes called the "triple deck" since there are three spatial scales: $O(\nu^{1/4})$ (critical layer),
$O(1)$ (shear flow) and $O(\nu^{-1/4})$ (recirculation area, rotational free). 

Note that when $\alpha$ is of order $\nu^{1/4}$, the critical layer is near the boundary, however when $\alpha$ is larger, it is "detached" from the
boundary. It is centered at $y_c$ which is of order $c$, but of size $\gamma^{-1}$ which is smaller with respect to $c$ if $\alpha \gg \nu^{1/4}$.
In this case the stream function and the velocity exhibit four scales: the size $O(\gamma^{-1})$ of the critical layer,
its distance $y_c$, of order $c$, from the boundary, the typical size of change of $U_s$, namely $O(1)$ and also $O(\alpha^{-1})$ which
corresponds to a "recirculation" size.


\subsection{Boundary Green function \label{GreenBoundary}}


We now look for the Green function of Orr Sommerfeld, together with its boundary condition at $y = 0$, under the form
$$
G = G^{int} + G^b,
$$
where $G^{int}$ has been previously constructed.
We look for $G^b(x,y)$ of the form
\beq \label{bl}
G^b(x,y) = d_s(x) \phi_{s,-}(y)  + d_f(x) {\phi_{f,-}(y)  \over \phi_{f,-}(0)}.
\eeq
We have the following result.

\begin{theorem} \label{boundG2}
We have
$$
d_s =  { \phi_{f,-}(0)  \over W^{-1}[\phi_{s,-},\phi_{f,-}] } O \Bigl( { \gamma \over \eps \mu^2(x)} \Bigr) 
= O \Bigl( { 1 \over \eps  \mu^2(x) | c - c(\alpha,\nu) |} \Bigr)
$$
and
$$
d_f =  {  \phi_{f,-}(0)  \over W^{-1}[\phi_{s,-},\phi_{f,-}] } O \Bigl( {1  + | \gamma c| \over \eps \mu^2(x)} \Bigr)
= O \Bigl( {1 + | \gamma c|  \over \eps \gamma \mu^2(x) | c - c(\alpha,\nu) |} \Bigr) .
$$
\end{theorem}

\Remark
Note that $G^b$ is singular when $W[\phi_{s,-},\phi_{f,-}] = 0$, namely at eigenvalues
of the Orr Sommerfeld operator, which is not the case of $G^{int}$ which is well defined everywhere.

\begin{proof}
Let us look for $d_s$ and $d_f$ such that
\beq \label{Greenb1}
G^{int}(x,0) + G^b(0) = 
\partial_y G^{int}(x,0) + \partial_y G^b(0) =  0.
\eeq
Let
$$
M = \left( \begin{array}{cc} \phi_{s,-}(0) & 1 \cr
\partial_y \phi_{s,-}(0) & \partial_y \phi_{f,-}(0) / \phi_{f,-}(0) \cr \end{array} \right) .
$$
Then (\ref{Greenb1}) can be rewritten as
$$
M d = - \Bigl( G^{int}(x,0), \partial_y G^{int}(x,0) \Bigr) 
$$
where $d = (d_s,d_f)$. Note that
$$
(G^{int}(x,0), \partial_y G^{int}(x,0))  = Q (a_+,b_+)
$$
with
$$
Q =  \left( \begin{array}{cc}
 \phi_{s,+}(0)  & \phi_{f,+}(0) / \phi_{f,+}(x) \cr
 \partial_y \phi_{s,+}(0)   & \partial_y \phi_{f,+}(0) / \phi_{f,+}(x) \cr
\end{array} \right) 
= \left( \begin{array}{cc}
 O(1)  & O(1) \cr
 O(1)   & O(\gamma) \cr
\end{array} \right) 
$$
where we have used that $\phi_{f,+}(0) / \phi_{f,+}(x)$ is bounded.
By construction we have
\beq \label{defid}
d = -   M^{-1} Q (a_+,b_+) 
\eeq
where
$$
M^{-1} = {\phi_{f,-}(0) \over W[\phi_{s,-},\phi_{f,-}](0) }   \left( \begin{array}{cc} \partial_y \phi_{f,-}(0) / \phi_{f,-}(0) & - 1 \cr
- \partial_y \phi_{s,-}(0)& \phi_{s,-}(0)  \end{array} \right).
$$
We recall that $\phi_{f,-}(0)$ is of order $O(1)$. Of course $M^{-1}$ is singular when $W[\phi_{s,-},\phi_{f,-}](0)  = 0$, namely near eigenvalues
of Orr Sommerfeld.

 As $\partial_y \phi_{f,-}(0) / \phi_{f,-}(0)$ is of order $O(\gamma)$ and $\phi_{s,-}(0)$ is of order $O(1)$, we have
$$
M^{-1} =  {\phi_{f,-}(0) \over W[\phi_{s,-},\phi_{f,-}] }   \left( \begin{array}{cc} O(\gamma) & - 1 \cr
O(1) & \phi_{s,-}(0) \end{array} \right) .
$$
As a consequence, as $\gamma \phi_{s,-}(0)$ is of order $\gamma c$,
$$
M^{-1} Q = {\phi_{f,-}(0) \over W[\phi_{s,-},\phi_{f,-}] }
\left( \begin{array}{cc} O(\gamma) & O(\gamma) \cr
O(1) & O(\gamma c) + O(1) \end{array} \right),
$$
thus, using (\ref{appGreen0}) together with $\gamma \lesssim \mu(x)$, we have
$$
d_s =   {\phi_{f,-}(0) \over W[\phi_{s,-},\phi_{f,-}]} O \Bigl( { \gamma \over \eps \mu^2(x)} \Bigr)
$$
and
$$
d_f =  {\phi_{f,-}(0)   \over W[\phi_{s,-},\phi_{f,-}]} O \Bigl( {1 + | \gamma c|  \over \eps \mu^2(x)}   \Bigr) .
$$
This ends the description of the Green function.
\end{proof}


\section{Adjoint of Orr Sommerfeld operator \label{adjointNS}}


The adjoint of Orr Sommerfeld operator with respect to the scalar product
$$
(\psi_1,\psi_2) = \int \psi_1 \bar \psi_2 dx .
$$
 is given by
\beq \label{Orrad}
Orr^t_{c,\alpha,\nu}(\psi) :=  (\partial_z^2 - \alpha^2) (U_s - \bar c)   \psi 
- U_s''  \psi  
+ { \nu \over i \alpha}   (\partial_z^2 - \alpha^2)^2 \psi,
\eeq
with boundary conditions $\psi(0) = \partial_z \psi(0) = 0$.

Let $\psi_1$ be an eigenvector of $Orr$ with corresponding eigenvalue $\lambda_1$ and
let $\psi_2$ be an eigenvector of $Orr^t$ with corresponding eigenvalue $\lambda_2$.
Then, multiplying $Orr(\psi_1) = \lambda_1 \psi_1$ by $\bar \psi_2$, $Orr^t(\psi_2) = \lambda_2 \psi_2$ by $\bar \psi_1$ and
combining both equalities, we get
\beq \label{ortho}
(\lambda_1 - \lambda_2) \int \psi_1 (\partial_z^2 - \alpha^2) \bar \psi_2 dx = 0.
\eeq
We can therefore normalize the eigenvectors $\psi_1$ and $\psi_2$ such that
\beq \label{ortho2} 
\int \psi_1 (\partial_z^2 - \alpha^2) \bar \psi_2 dx = \delta_{\lambda_1 = \lambda_2}.
\eeq
Let $v_1 = \nabla^\perp \psi_1$ and $v_2 = \nabla^\perp \psi_2$, then we get
\beq \label{ortho3}
\int v_1 \bar v_2 dx =  \delta_{\lambda_1 = \lambda_2}.
\eeq
We refer to \cite{Bian1} and  \cite{Reid}  for a formal study of this adjoint operator.
All the computations done in \cite{Bian1} can be justified using the methods developed in the current paper. 
We just recall the main results.

First the adjoint of  operator ${\cal A}$ is simply $\partial_y^2 \circ Airy$. Thus the fast modes $\phi_{f,\pm}^t$ are well approximated 
by the Airy functions $Ai$ and $Ci$. 
Moreover, $\phi_{s,-}^t(y)$ is a perturbation of $\psi_{-,\alpha}^t$ which is approximately constant. This leads to 
$$
\phi_{s,-}^t(y) = e^{- \alpha y} - f(y_c) \psi(y) - g(y) + O(\gamma^{-1})
$$
where
$$ 
f(y) = - 2 \alpha Ray_{\alpha,c}^{-1} \Bigl( (U_s - c) U_s' e^{- \alpha y} \Bigr) ,
$$
$$
g(y) = {f(y) - f(y_c) \over U_s(y) - c} 
$$
and
$$
\psi(y) = Airy^{-1} e^{-\alpha y} .
$$
Thus the eigenmode $\psi_{s,-}^t$, used in the projection of the initial data, is completely explicit.


\section{Green functions}



\subsection{Full description of $G_{\alpha,c,\nu}$ \label{fullGreen}}


Let us combine the previous results to fully describe the Green function of Orr Sommerfeld equation.

\begin{theorem}
The Green function $G_{\alpha,c,\nu}$ of Orr Sommerfeld equation is of the form
$$
G_{\alpha,c,\nu} = G^{int} + G^b
$$
where
\beq \label{int}
G^{int}(x,y) = a_+(x)  \phi_{s,+}(y) 
+ {b_+(x) }  {\phi_{f,+}(y) \over \phi_{f,+}(x)} 
\quad \hbox{for} \quad y < x,
\eeq
$$
G^{int}(x,y) = a_-(x)  \phi_{s,-}(y)
+ {b_-(x)} {\phi_{f,-}(y) \over \phi_{f,-}(x)} 
\quad \hbox{for} \quad y  > x,
$$
$$
G^b(x,y) = d_s(x) \phi_{s,-}(y)  + d_f(x) {\phi_{f,-}(y)  \over \phi_{f,-}(0)},
$$
 with
$$
a_\pm(x) = O \Bigl( {1 \over \eps \mu^2(x)} \Bigr), \qquad 
b_\pm (x)= O \Bigl( {1 \over \eps \mu^3(x)} \Bigr),
$$
$$
d_s = O \Bigl( { 1 \over \eps  \mu^2(x) | c - c(\alpha,\nu) |} \Bigr)
$$
and
$$
d_f = O \Bigl( {1 + | \gamma c|  \over \eps \gamma \mu^2(x) | c - c(\alpha,\nu) |} \Bigr) .
$$
\end{theorem}

\Remark
In the particular case where $\alpha$ is of order $\nu^{1/4}$, then 
$$
a_\pm(x) = O(\nu^{-1/4}), \qquad b_\pm(x) = O(1)$$
and
$$
d_s = O (  \nu^{-1/4}  | c - c(\alpha,\nu) |^{-1} ), \qquad
d_f = O (   | c - c(\alpha,\nu) |^{-1} ).
$$


\subsection{Proof of Theorem \ref{maintheo} } 


As for Theorem \ref{growthRayleigh}, we study the Green function of the linearized Navier Stokes equation, namely the solution
$G_\alpha(t,x,y)$ to
$$
\partial_t \omega + L_{NS}^\omega \omega = 0 
$$
with initial data $\omega(0,\cdot) = \delta_x$. In this equation, $L_{NS}^\omega$ is the linearized Navier Stokes operator in vorticity formulation,
namely
$$
L_{NS}^\omega = \hbox{curl} \, L_{NS} \nabla^{\perp} \Delta^{-1} .
$$
As $L_{NS}^\omega$ is a sectorial operator, we can use Dunford's formula to express the solution, namely
\beq \label{Dun1}
\psi(t) = {1 \over 2 i \pi } \int_\Gamma e^{\lambda  t} OS_{\alpha,c,\nu}^{-1} \, \omega_0 \, d\lambda ,
\eeq
where the contour is "on the right" of the spectrum.
We also recall that
\beq \label{invv1}
OS_{\alpha,c,\nu}^{-1} \, \omega_0 = \int G_{\alpha,c,\nu}(x,y) \omega_0(x) \, dx ,
\eeq
where $G_{\alpha,c,\nu}$ is the Green function detailed in the previous section.
Thus, combining (\ref{Dun1}) and (\ref{invv1}), we get
\beq \label{invv2}
G_\alpha(t,x,y) = - {\alpha \over 2  \pi} \int_\Gamma e^{i \alpha c t}  G_{\alpha,c,\nu}(x,y) \, dc ,
\eeq
where $\Gamma$ is now a contour "above" the spectrum.

We note that $G_{\alpha,c,\nu}$ is holomorphic except at its pole, namely at the eigenvalue $c(\alpha,\nu)$.
We can thus move the contour $\Gamma$ downwards in the region $\Im c < 0$. 
Using Theorem \ref{boundG}, we can move $\Gamma$ till $\Im c$ reaches 
$- C_0 | \gamma |^{-1}$, for some small positive constant $C_0$. 
Doing so we leave aside the only  pole of $G_{\alpha,c,\nu}$ at $c(\alpha,\nu)$ such that
$\Re c(\alpha,\nu) \ge - C_0 \gamma^{-1}$.
This leads to
\beq \label{invv3}
G_\alpha(t,x,y) = - i \alpha  Res\Bigl( G_{\alpha,c,\nu}(x,y), c(\alpha,\nu) \Bigr) e^{ \lambda(\alpha,\nu) t}
- {\alpha \over 2  \pi} \int_\Gamma e^{i \alpha c t}  G_{\alpha,c,\nu}(x,y) \, dc ,
\eeq
where now $\Gamma$ is the contour
$$
\Gamma =  \Bigl\{ - A - |\gamma|^{-1} i + (1 + i) \rit_- \Bigr\} \cup
[-A - |\gamma|^{-1} i , A - |\gamma|^{-1} i ]
\cup   \Bigl\{ A - |\gamma|^{-1} i + (1 + i) \rit_+ \Bigr\},
$$
and where $Res$ stands for the residue.
We define 
$$
\phi(x,y) =  - i \alpha Res\Bigl( G_{\alpha,c,\nu}(x,y), c(\alpha,\nu) \Bigr).
$$
As $G^{int}$ has no pole, we have in fact
$$
\phi_{\alpha,\nu}(x,y) =  - i \alpha Res\Bigl( G^b(x,y), c(\alpha,\nu) \Bigr),
$$
leading to
$$
P_{\alpha,\nu} v_{\alpha,0}(x) = - i \alpha \int_\Gamma \phi_{\alpha,\nu}(x,y) \omega_{\alpha,0} (x)  \, dx,
$$
which can be rewritten using the eigenvalue of the dual operator. 

The real part of $i \alpha c$ on that contour $\Gamma$ is  
$$
\theta(\nu) = | \alpha  \gamma^{-1} | \sim \nu^{1/3} \alpha^{2/3}.
$$
Choosing $C_0$ large enough, we may assume that $| c - c(\alpha,\nu)| \gtrsim | \gamma |^{-1}$.
We split $G_{\alpha,c,\nu}(x,y)$ into $G^{int}$ and $G^b$ and further split $G^{int}$.
We have
$$
\int_\Gamma G^b(x,y) \, dc = {\alpha \over 2 \pi} \int_\Gamma e^{i \alpha c t} d_s(x) \phi_{s,-}(y) \, dc
+  {\alpha \over 2 \pi} \int_\Gamma e^{i \alpha c t} d_f(x) \phi_{f,-}(y) \, dc . 
$$
We recall that $d_s$, $d_f$, $\phi_{s,-}$ and $\phi_{f,-}$ all depend on $c$.
We note that $\phi_{s,-}(y)$ is bounded by $e^{- \alpha y}$, $\phi_{f,-}(y)$ is bounded by $e^{- | \gamma | y}$ 
and $d_s(x)$ and $d_f(x)$ are given by Theorem \ref{boundG}, which gives
$$
\Bigl| \int_\Gamma G^b(x,y) \, dc \Bigr| \lesssim {| \alpha \gamma | \over | \eps \mu^2(x) | } e^{- \theta(\nu) t}  e^{- |\alpha| y} .
$$
We now turn to $G^{int}$. If $y < x$ 
$$
\int_\Gamma G^{int}(x,y) \, dc = {\alpha \over 2 \pi} \int_\Gamma e^{i \alpha c t} a_+(x) \phi_{s,+}(y) \, dc
+  {\alpha \over 2 \pi} \int_\Gamma e^{i \alpha c t} b_+(x) {\phi_{f,+}(y) \over \phi_{f,+}(x)} \, dc ,
$$
thus
$$
\Bigl| \int_\Gamma G^{int}(x,y) \, dc \Bigr| \lesssim {| \alpha  | \over | \eps \mu^2(x) | } e^{- \theta(\nu) t}  e^{+|\alpha| y} .
$$
The result is similar if $y > x$ with $e^{- | \alpha | y}$ instead of $e^{+ | \alpha | y}$.

We observe that 
$$
{| \alpha \gamma | \over | \eps \mu^2(x) |} \le { | \alpha \gamma | \over | \eps \gamma^2 |} \le | \alpha \gamma^2|.
$$
This ends the proof of Theorem \ref{maintheo}.

\medskip

\Remark
The vorticity enjoys a better decay in its slow part, but reveals the vorticity of the Airy solution, which leads to 
$$
\Bigl| \int_\Gamma (\partial_y^2 - \alpha^2) G^b(x,y) \, dc \Bigr| \lesssim {| \alpha | \over | \eps \mu(x) |} 
e^{- \alpha | \gamma | t} \Bigl[  e^{- \beta y} + \mu^2(x)  e^{- | \gamma| | y  |} \Bigr].
$$


\section{Appendix \label{appendix}}



\subsection{Airy and Tietjens functions \label{appendix1}}


Let $Ai$ and $Bi$ be the classical Airy functions. 
We recall that, as $x$ goes to $\pm \infty$, for $|\arg(x)| < \pi$, 
\beq \label{asymptAi}
Ai(x) \sim {1 \over 2 \sqrt{\pi}}  { e^{- 2 x^{3/2}  / 3} \over x^{1/4}} \Big( 1 + O(|x|^{-3/2})\Big)
\eeq
 and, for $| \arg(x) | < \pi/3$,
\beq \label{asymptBi}
Bi(x) \sim {1 \over  \sqrt{\pi}}  { e^{2 x^{3/2}  / 3} \over x^{1/4}} \Big( 1 + O(|x|^{-3/2})\Big).
\eeq
The behavior of $Ai(x)$ thus depends on the argument of $x$. For instance if $x$ is real then $Ai(x)$ decays more than
exponentially fast as $x \to + \infty$. On the contrary, $Ai(x)$ is oscillatory as $x \to - \infty$. In this paper we are interested of
$x$ with argument $\pi / 6$ or $- 5 \pi / 6$. If $\arg(x) = \pi / 6$, then $Ai(x) \to 0$ as $x \to + \infty$ in an exponentially and oscillatory way.
On the contrary, if $\arg(x) = - 5 \pi / 6$ then $Ai(x)$ diverges as $x \to - \infty$. It is oscillatory and its modulus goes to $+ \infty$ in an exponential way.
We refer to \cite{Reid} for the discussion of the associated notion of Stokes and anti-Stokes lines.

Let us introduce the following combination of $Ai$ and $Bi$
\beq \label{defiCi}
Ci = - i \pi (Ai + i Bi).
\eeq
We also introduce the first and second primitives of $Ai$ and $Ci$, denoted by $Ai(x,1)$, $Ai(x,2)$, $Ci(x,1)$, $Ci(x,2)$.
For large $x$, we have, using the results of Appendix \ref{appendix3},
\beq \label{asymptAi1}
Ai(x,1) \sim - {1 \over 2 \sqrt{\pi}}  { e^{- 2 x^{3/2}  / 3} \over  x^{3/4}} \Big( 1 + O(|x|^{-3/2})\Big),
\eeq
\beq \label{asymptAi2}
Ai(x,2) \sim {1 \over 2 \sqrt{\pi}}  { e^{- 2 x^{3/2}  / 3} \over x^{5/4}} \Big( 1 + O(|x|^{-3/2})\Big)
\eeq
and similarly for $Ci(x,1)$ and $Ci(x,2)$. In particular, for large $x$, we have
\beq \label{asymptAi2}
{Ai(x) \over Ai(x,1)} \sim x^{1/2}, \qquad
{Ai(x) \over Ai(x,2)} \sim x, \qquad 
{Ai'(x) \over Ai(x,2)} \sim x^{3/2}.
\eeq
We will also need to compare $Ai(x)$ and $Ci(x)$ together with their integrals. Let $Ai^{abs}(x,1)$ be the primitive of $| Ai(x) |$ defined by
$$
Ai^{abs}(x,1) = \int_x^{+\infty} | Ai(x e^{i \pi / 6}) | \, dx
$$
and let
$$
Ci^{abs}(x,1) = \int_{- \infty}^x | Ci(x e^{i \pi / 6}) | \, dx .
$$
Then
\beq \label{compareintegrale}
Ai^{abs}(x,1)  \lesssim  x^{-1/2} |Ai(x)| \qquad Ci^{abs}(x,1) \lesssim x^{-1/2} | Ci(x) | .
\eeq
We now turn to the study of the so called Tietjens function.
Following classical notations (see \cite{Reid}, page $273$), we introduce as previously
$$
z = \Bigl( {\alpha U_c' \over \nu} \Bigr)^{1/3} y_c, \qquad \xi_1 = - i^{1/3} z
$$
and the Tietjens function defined by 
$$
Ti(z) = {Ai(\xi_1,2) \over \xi_1 Ai(\xi_1,1) } .
$$
A detailed study of Tietijens function may be found in \cite{Reid}, pages $273$ to $276$.
We recall the corresponding results here. 
In particular, as $|z| \to + \infty$ with $| \arg(z) | < \pi$, using the asymptotic expansion of $Ai$, we obtain
\beq \label{Tietjensinfini}
Ti(z) = - {e^{i \pi / 4} \over z^{3/2}} \Bigl( 1 + {5 \over 4} {e^{i \pi /4} \over z^{3/2}}
+ {151 \over 32} {e^{i \pi /2} \over z^3} + ... \Bigr) .
\eeq
Note that the leading term of (\ref{Tietjensinfini}) may be obtained using (\ref{asymptAi1}) and (\ref{asymptAi2}).

Near $z = 0$, $Ti(z)$ has a pole
\beq \label{Tietjenszero}
Ti(z) = 3 Ai'(0) e^{5 i \pi/6} {1 \over z} + 1 - {3 \sqrt{3} \over 2 \pi} + O(z) .
\eeq
Moreover, as $z \to 0$, $z$ being real, 
$$
\Re Ti(z) + \sqrt{3} \Im Ti(z) \to 1 - {3 \sqrt{3} \over 2 \pi} .
$$
Numerically we observe that, when $z$ is real and positive, $\Im Ti(z)$ vanishes as $| z| \to + \infty$ and also when $z = z_0$ with $z_0 \sim 2.297$.
Near that point,
\beq \label{Tietjensother}
Ti(z) = Ti(z_0) + Ti'(z_0) (z - z_0) + O( (z - z_0)^2)
\eeq
where 
$$
Ti(z_0) \sim 0.5645, \qquad Ti'(z_0) \sim -0.1197 + 0.2307 i.
$$


\subsection{Approximate solvers \label{appendix2}}


In this paper, we often encounter the following situation. We want to solve the equation
\beq \label{eqA}
{\cal A} \phi = f
\eeq
but only know how to solve an approximate equation ${\cal A}_1 \phi = f$.
Let ${\cal A}_2 = {\cal A} - {\cal A}_1$. Let $G_1$ be the Green function of ${\cal A}_1$, namely be the solution to
$$
{\cal A}_1 G_1(x,\cdot) = \delta_x .
$$
Then an approximate solver ${\cal S}_1$ of (\ref{eqA}) is given by
$$
\phi_1(y) = {\cal S}_1 f(y) = \int G_1(x,y) f(x) \, dx .
$$
Note that ${\cal A}_1 \phi_1 = f$, therefore
$$
{\cal A}\phi_1 = {\cal A}_2 \phi_1 + f.
$$
Let us define  the error of approximation ${\cal R}$ by 
$$
{\cal R} f = {\cal A}_2 \phi_1 =  {\cal A}_2 \, {\cal S}_1 \, f.
$$
In our case, ${\cal R}$ will have a kernel $R(x,y)$, namely
$$
({\cal R} f) (y) = \int R(x,y) f(x) \, dx .
$$

\begin{proposition} \label{prophypo2}
Assume either that
\beq \label{hypo1}
(A1) \qquad \sup_y \int | R(x,y) | \, dx < 1 
\eeq
or that there exists some positive function $\theta$ such that
\beq \label{hypo2} 
(A2) \qquad  \int |  R(x,y) |  \theta(x) \, dx \le k \theta(y)
  \eeq
  for some $k < 1$ and for every $y$.
  Then under assumption (A1) or (A2), there exists a solution to  (\ref{eqA}).
  \end{proposition}
  
  \begin{proof}
Assume that ${\cal R}$ is contractant in some space $X$, and that ${\cal S}_1$ is continuous from this space $X$ to another space $Y$
and introduce the iterative scheme
$$
\phi_n =  {\cal S}_1 f_n,
$$
$$
f_n = {\cal R} f_{n-1}
$$
starting with $f_0 = f$.
Then $f_n$ converges geometrically to $0$ in $X$ and  $\sum_n \phi_n$ converges in $Y$ to a solution $\phi$ of (\ref{eqA}).

The first assumption (A1) implies that ${\cal R}$ is contractant in $L^\infty$ and the second assumption (A2) that it is contractant in
$L^\infty(\theta \, dx)$ the space of functions $f$ such that $f(x) / \theta(x) \in L^\infty$.
\end{proof}

Let us give two examples of applications of this Proposition.


\subsubsection*{First example: contraction in $L^\infty$}


 Let $g(x)$ be a smooth function. To solve
\beq \label{ex1}
\eps^2 \partial_y^2 \phi - \phi + \eps g \phi = 0
\eeq
we introduce ${\cal A}_1 \phi = \eps^2 \partial_y^2 \phi - \phi$ and ${\cal A}_2 \phi = \eps g \phi$. Then
$G_1(x,y) = (2\eps)^{-1} e^{- | x - y | / \eps}$
and
$$
R(x,y) = ({\cal R} \delta_x) (y) = {g(y) e^{- | x - y | / \eps} \over 2} .
$$
Thus, if $g$ is bounded, $\| R(x,\cdot) \|_{L^1} \ll \eps$, and  ${\cal R}$ is contractant in $L^\infty$ provided $\eps$ is small enough. 
This strategy will be used to prove Proposition \ref{Airy0}.


\subsubsection*{Second example involving weights}


As a second example, we study
\beq \label{ex2}
\eps^2 \partial_y^4 \phi - \partial_y^2 \phi =  g \phi,
\eeq
which is a model for the equation solved in Proposition \ref{constructfast}.
We rewrite (\ref{ex2}) as
\beq \label{ex3}
\eps^2 \partial_y^2 \psi - \psi =  g \partial_y^{-2} \psi
\eeq
where $\psi = \partial_y^2 \phi$, and where $\partial_y^{-2} \psi$ is the second primitive, calculated from $+ \infty$.
The Green function for $\psi$ is $G_1$ and the error term is
$$
({\cal R} \delta_x)(y) = g(y) \partial_y^{-2} G_1(x,y) .
$$
Up to terms of order $O(\eps)$, $\partial_y^{-2} G_1(x,y)$ equals $(x - y)_+$. 

Let $\theta$ be a positive function, like $e^{- \beta y}$ for some positive $\beta$.
Let $X = L^{\infty}(\theta(x) dx)$ be the set of functions $f$ such that $f / \theta$ is bounded. Let $f \in X$. Then
$$
({\cal R}f) (y) = g(y)  \int \partial_y^{-2} G_1(x,y) f(x) \, dx.
$$
If we assume that, for some constant $k < 1$, we have, for any $y$,
\beq \label{contra}
I =  | g(y) | \int |  \partial_y^{-2} G_1(x,y) | \theta(x) \, dx \le k \theta(y),
\eeq
then ${\cal R}$ is a contraction in $X$ as desired.
In the case where $g(y)$ is bounded and $\theta(x) = e^{- \mu x}$, 
$$
I \le \| g \|_{L^\infty} \int_y^{+\infty} (x - y) e^{-\mu x} \, dx  \le {C \over \mu^2} e^{- \mu y} .
$$
Thus the assumption (\ref{contra}) is checked provided $\mu$ is large enough, which leads to the existence of a solution to (\ref{ex2}).
The key point of the proof is that the error is of zeroth order, and thus much smaller than the source itself.

 
 \subsection{Integrals of rapidly decaying functions \label{appendix3}}
 

 In the study of integrals of Airy functions we will use the following classical result.

 \begin{lemma}
 Let $h(x)$ be a given $C^2$ function such that $\Re h(x) \to 0$ as $x \to + \infty$, and let $\mu$ be a number. Then
  if $h'' / h'^2 \to 0$ as $x \to + \infty$ and if $\int_\rit \exp(\mu \Re(y)) \, dy < + \infty$, then
 $$
   \int_x^{+\infty} e^{\mu h(y)} \, dy \sim {e^{\mu h(x)} \over \mu h'(x)}.
$$
\end{lemma}

\begin{proof}
We have
 \beq \label{derivvv}
 \partial_x \Bigl( {e^{\mu h(x)} \over \mu h'(x)} \Bigr) = e^{\mu h(x)} - {e^{\mu h(x)} h''(x) \over \mu h'^2(x)} .
 \eeq
 In particular
 \beq \label{derivvv2}
 \Bigl| \int_x^{+\infty} e^{\mu h(y)} \, dy - {e^{\mu h(x)} \over \mu h'(x)} \Bigr| \le 
 \mu^{-1} \int_x^{+\infty}  {e^{\mu \Re h(y)} | h''(y) | \over  |h'^2(y)| }  \, dy .
 \eeq
 We may apply again (\ref{derivvv}) to $\Re h(x)$ which gives
 \beq \label{derivvv3}
 \Bigl| \int_x^{+\infty} e^{\mu \Re h(y)} \, dy - {e^{\mu \Re h(x)} \over \mu \Re h'(x)} \Bigr| \le 
 \mu^{-1} \int_x^{+\infty}  {e^{\mu \Re h(y)} | \Re h''(y) | \over  | \Re h'(y)|^2 }  \, dy .
 \eeq 
We then combine (\ref{derivvv2}) and (\ref{derivvv3}) to obtain the results.
\end{proof}
   
If $h'' / h'^2$ is only bounded, then we still get that the integral is of order $e^{\mu h(x)} / \mu h'(x)$ up to a factor of order $O(1)$.



\end{document}